\documentclass{amsart}
\usepackage[OT2,T1]{fontenc}
\usepackage[dvips]{graphicx}
\usepackage{hyperref}
\usepackage{amsmath}%
\usepackage{amsfonts}%
\usepackage{amssymb}%
\usepackage{graphicx}
\usepackage{txfonts}
\usepackage{calrsfs}
\newtheorem{theorem}{Theorem}
\newtheorem{corollary}{Corollary}

\newtheorem{example}{Example}
\newtheorem{lemma}{Lemma}
\newtheorem{proposition}{Proposition}
\newtheorem{remark}{Remark}
\newcommand{\eps}{\varepsilon}
\DeclareMathAlphabet{\mathpzc}{OT1}{pzc}{m}{it}

\DeclareMathOperator{\Graph}{Gr}

\DeclareMathOperator{\co}{co}

\DeclareMathOperator{\modulus}{mod}
\DeclareMathOperator{\supp}{supp}

\newcommand{\?}{\vspace{\baselineskip}}

\newcommand{\w}{\tilde}
\newcommand{\map}{\multimap}
\newcommand{\<}{\leqslant}
\newcommand{\n}{{n\geqslant 1}}

\newcommand{\K}{{k\geqslant 1}}
\newcommand{\R}[1]{\varmathbb{R}^{#1}}
\newcounter{Page}

\begin{document}
\title[Structure of the solution set to Volterra integral inclusions and applications]{Structure of the solution set to Volterra integral inclusions and applications} 
\author{Rados\l aw Pietkun}
\address{Toru\'n, Poland}
\email{rpietkun@pm.me}
\begin{abstract}
The topological and geometric structure of the solution set to Volterra integral inclusions in Banach spaces is investigated. It is shown that the set of solutions in the sense of Aumann integral is nonempty compact acyclic in the space of continuous functions or is even an $R_\delta$-set provided some appropriate conditions on the Banach space are imposed. Applications to the periodic problem for this type of inclusions are given.
\end{abstract}
\maketitle

\section{Introduction}
\par A key issue in the description of the geometry of the set of solutions of the differential inclusion \[\dot{x}(t)\in F(t,x(t))\] is the possibility to free ``switching'' between trajectories within this set by selecting as the starting point of the new solution any point of evaluation of the current solution. A particularly clear illustration of this property is the solution of the differential equation \[\dot{x}(t)=f(t,x(t)),\] whose integral description \[x(t)=x(\tau)+\int_\tau^tf(s,x(s))\,ds\] remains true for any point $\tau$ in a specific time interval. But the situation is complicated even in the case of the simplest integral equations enriched by the additional variable of time. Let us visualize this difficulty by the following example of Volterra-type equation
\begin{equation}\label{1}
x(t)=\int_0^tk(t,s)f(s,x(s))\,ds,\;t\geqslant 0.
\end{equation}
The differential case seems to suggest that the solution of the equation \eqref{1}, ``starting'' at the time $\tau>0$ from the point $x(\tau)$ should be presented as follows \[y(t)=x(\tau)+\int_\tau^tk(t,s)f(s,y(s))\,ds,\;t\geqslant\tau.\] However, the mapping $z$ defined by the formula
\[z(t)=
\begin{cases}
x(t)&\mbox{for }t\<\tau,\\
y(t)&\mbox{for }t\geqslant\tau,\\
\end{cases}\]
is not only a continuation of the solution $x$, but in principle does not satisfy the equation \eqref{1}. Indeed, if we assume that \[z(t)=\int_0^tk(t,s)f(s,z(s))\,ds,\] then for $t>\tau$ there must be equality \[\int_0^tk(t,s)f(s,z(s))\,ds=x(\tau)+\int_\tau^tk(t,s)f(s,y(s))\,ds.\] 
Using the definition of $z$ we obtain the dependence 
\begin{align*}
\int_0^\tau k(t,s)f(s,x(s))\,ds+\int_\tau^tk(t,s)f(s,y(s))\,ds=\int_0^\tau k(\tau,s)f(s,x(s))\,ds+\int_\tau^tk(t,s)f(s,y(s))\,ds,
\end{align*}
which amounts to \[\int_0^\tau k(t,s)f(s,x(s))\,ds=\int_0^\tau k(\tau,s)f(s,x(s))\,ds.\] This equality is obviously false in general, non-trivial case of the kernel $k$ dependent on a parameter, which is not subjected to integration. This seemingly simple observation shows however, according to the author, the source of misunderstandings and erroneous constructions used in the previously published results concerning the geometric structure of the solution set of integral inclusions. The wrong approach described above can be found both in the older literature demonstrating connectedness of the set of solutions (\cite[Th.1]{bulg1}, \cite[Th.3]{bulg2}) as well as newer, justifying acyclicity (\cite[Th.4.11]{gelman}) or the $R_\delta$-property (\cite[Th.1.2.2]{agarwal}). It seems that the obvious way out of this situation is to substitute in place of the mapping $y$ the solution of the integral equation of the form \[y(t)=\int_0^\tau k(t,s)f(s,x(s))\,ds+\int_\tau^tk(t,s)f(s,y(s))\,ds,\;t\geqslant\tau.\]
\par However, adoption of this method involves additional technical difficulties, accompanying the need to justify the convergence of a sequence of integral inclusion perturbations in the functional space in place of a usual pointwise convergence of solutions subjected to evaluation. Overcoming these difficulties forced the adoption of stronger assumptions $(k_1)$-$(k_4)$ on the kernel $k$ than those generic ones used in the cited publications (assumptions of the form $(k_5)$-$(k_6)$). These assumptions not only ensure the correctness of the definition and the continuity of the Volterra integral operator, but also guarantee the uniqueness of the description of solutions to the inclusion, whose right side is a multivalued Aumann integral. It remains an open question to what extent is this choice optimal for a description of such geometric properties of the solution set as acyclicity.
\section{Preliminaries}
Let $(E,|\cdot|)$ be a Banach space. For any $\eps>0$ and $A\subset E$, $B(A,\eps)$ ($D(A,\eps)$) is an open (closed) $\eps$-neighbourhood of the set $A$. The closure and the closed convex envelope of $A$ will be denoted by $\overline{A}$ and $\overline{\co} A$, respectively. The (normed) space of bounded linear endomorphisms of $E$ is denoted by $\mathcal{L}(E)$ and $E^*$ stands for the normed dual of $E$. Given $S\in\mathcal{L}(E)$, $||S||_{{\mathcal L}}$ is the norm of $S$.\par For $a, b\in\R{}$, $(C([a,b],E),||\cdot||)$ is the Banach space of continuous maps $[a,b]\to E$ equipped with the maximum norm. Let $1\<p<\infty$. By $(L^p([a,b],E),||\cdot||_p)$ we mean the Banach space of all (Bochner) $p$-integrable maps $f\colon[a,b]\to E$, i.e. $f\in L^p([a,b],E)$ iff f is strongly measurable and \[||f||_p=\left(\int_a^b|f(t)|^p\,dt\right)^{\frac{1}{p}}<\infty.\] Recall that strong measurability is equivalent to the usual measurability in case $E$ is separable. In what follows we shall make a frequent use of this generalization of the well-known Dunford-Pettis weak compactness criterion.
\begin{theorem}{\em (\cite[Prop.11]{ulger})}\label{2}
Suppose the function $g\colon[a,b]\to\R{}$ is integrable and the sequence $(f_n)_\n$ in $L^1([a,b],E)$ is such that:
\begin{itemize}
\item[(i)] for almost all $t\in[a,b]$ and for all $\n$, $|f_n(t)|\<g(t)$,
\item[(ii)] for almost all $t\in[a,b]$, the sequence $(f_n(t))_\n$ is relatively weakly compact.
\end{itemize}
Then the sequence $(f_n)_\n$ is relatively weakly compact.
\end{theorem}
\par Given metric space X, a set-valued map $F\colon X\map E$ assigns to any $x \in X$ a nonempty subset $F(x)\subset E$. $F$ is (weakly) upper semicontinuous, if the small inverse image $F^{-1}(A)=\{x\in X\colon F(x)\subset A\}$ is open in $X$ whenever $A$ is (weakly) open in $E$. A map $F\colon X\map E$ is upper hemicontinuous if for each $p\in E^*$, the function $\sigma(p,F(\cdot))\colon X\to\R{}\cup\{+\infty\}$ is upper semicontinuous (as an extended real function), where $\sigma(p,F(x))=\sup_{y\in F(x)}\langle p,y\rangle$. It is clear that if a set-valued map is weakly upper semicontinuous, then it is upper hemicontinuous. Conversely (see \cite[Th.1.4.2]{aubin}), if a multivalued map is upper hemicontinuous and has weakly compact convex values, then it is weakly upper semicontinuous. We have the following characterization (\cite[Prop.2(b)]{bothe}): a map $F\colon X\map E$ with convex values is weakly upper semicontinues and has weakly compact values iff given a sequence $(x_n,y_n)$ in the graph $\Graph(F)$ with $x_n\to x$ in $X$, there is a subsequence $y_{k_n}\rightharpoonup y\in F(x)$ ($\rightharpoonup$ denotes the weak convergence).
\par The following property, known as the convergence theorem, of upper hemicontinuous maps with convex values shall play a crucial role in section 3. of this paper.
\begin{theorem}\label{convergence}
Suppose the multivalued map $F\colon E\map E$ with closed convex values is upper hemicontinuous. If $I$ is a finite interval of $\,\R{}$ and sequences $(x_n\colon I\to E)_\n$ and $(y_n\colon I\to E)_\n$ satisfy the following conditions
\begin{itemize}
\item[(i)] $(x_n)_\n$ converges almost everywhere to a function $x\colon I\to E$,
\item[(ii)] $(y_n)_\n$ converges weakly in the space $L^1(I,E)$ to a function $y\colon I\to E$,
\item[(iii)] $y_n(t)\in\overline{\co}B(F(B(x_n(t),\eps_n)),\eps_n)$ for almost all $t\in I$, where $\eps_n\to 0^+$ as $n\to\infty$,
\end{itemize}
then $y(t)\in F(x(t))$ for almost all $t\in I$.
\end{theorem}
\begin{remark}
The thesis of this theorem remains true under the assumption that $(y_n)_\n$ converges weakly to $y$ in the space $L^p(I,E)$ for $p\in(1,\infty)$. Justification of this observation is not possible without an immersion in the proof of convergence theorem, unless we are able to demonstrate weak convergence of the sequence $(y_n)_\n$ in the space $L^1(I,E)$. It can be done under the additional assumption on the Banach space $E$. Namely, if $E^*$ has the Radon-Nikod\'ym property (this is the case for instance if $E$ is reflexive), then there is an isometrically isomorphic embedding $L^1(I,E)^*\hookrightarrow L^p(I,E)^*$, in view of the duality theorem (\cite[Th.IV.1.1]{uhl}). In that sense weak convergence in $L^p(I,E)$ results in weak convergence in $L^1(I,E)$.
\end{remark}
\par An upper semicontinuous map $F\colon E\map E$ is called acyclic if it has compact acyclic values. A set-valued map $F\colon E\map E$ is admissible (compare \cite[Def.40.1]{gorn}) if there is a metric space $X$ and two continuous functions $p\colon X\to E$, $q\colon X\to E$ from which $p$ is a Vietoris map such that $F(x)=q(p^{-1}(x))$ for every $x\in E$. Clearly, every acyclic map is admissible. Moreover, the composition of admissible maps is admissible (\cite[Th.40.6]{gorn}). In particular the composition of two acyclic maps is admissible. 
\par A real function $\gamma$ defined on the family of bounded subsets of $E$ is called a measure of non-compactness (MNC) if $\gamma(\Omega)=\gamma(\overline{\co}\Omega)$ for any bounded subset $\Omega$ of $E$. The following example of MNC is of particular importance: given a bounded $\Omega\subset E$,
\[\beta(\Omega):=\inf\{\eps>0:\Omega\mbox{ admits a finite covering by balls of a radius }\eps\}\]
is the Hausdorff MNC. Recall that this measure is regular, i.e. $\beta(\Omega)=0$ iff $\Omega$ is relatively compact; monotone, i.e. if $\Omega\subset\Omega'$ then $\beta(\Omega)\<\beta(\Omega')$ and non-singular, i.e. $\beta({a}\cup\Omega)=\beta(\Omega)$ for any $a\in E$ (for details see \cite{sadovski}). Concerning behaviour of the Hausdorff MNC towards the integration process we have the following result (direct consequence of \cite[Cor.3.1]{heinz}).
\begin{theorem}\label{3}
Suppose that the set of functions $\{w_n\}_\n\subset L^1([a,b],E)$ is integrably bounded by a function $\mu\in L^1([a,b],\R{})$. Then the mapping $t\mapsto\beta(\{w_n(t)\}_\n)$ is Lebesgue integrable and for all $t\in [a,b]$ the following estimations hold:
\begin{itemize}
\item[(i)] $\beta\left(\left\{\int_0^tw_n(t)\,dt\right\}_\n\right)\<2\int_0^t\beta(\{w_n(t)\}_\n)\,dt$ if $E$ is an arbitrary Banach space,
\item[(ii)] $\beta\left(\left\{\int_0^tw_n(t)\,dt\right\}_\n\right)\<\int_0^t\beta(\{w_n(t)\}_\n)\,dt$ if $E$ is a separable Banach space.
\end{itemize}
\end{theorem}
\par A set-valued map $F\colon E\map E$ is condensing relative to MNC $\gamma$ (or $\gamma$-condensing) provided, for every bounded $\Omega\subset E$, the set $F(\Omega)$ is bounded and $\gamma(\Omega)\<\gamma(F(\Omega))$ implies relative compactness of $\Omega$. \par Recall that for $\Omega$ bounded in $C([a,b],E)$ the expression
\[\modulus_C(\Omega)=\lim\limits_{\xi\to 0^+}\sup\limits_{x\in\Omega}\max\limits_{|t-\tau|\<\xi}|x(t)-x(\tau)|\]
defines a MNC on the space $C([a,b],E)$ (\cite[Ex.2.1.2.]{zecca}). Suppose that:
\begin{equation}\label{measure}
\nu_L(\Omega)=\max\limits_{D\in\Delta(\Omega)}\left(\sup\limits_{t\in[a,b]}e^{-Lt}\beta(D(t)),\modulus_C(D)\right),
\end{equation}
where $\Omega$ is bounded in $C([a,b],E)$, $L\in\R{}$ and $\Delta(\Omega)$ stands for the family of countable subsets of $\Omega$. So defined function $\nu_L$ constitutes a MNC on the space $C([a,b],E)$ with values in the closed cone $\R{2}_+$ (\cite[Ex.2.1.4.]{zecca}). It is easy to see that the measure $\nu_L$ has all the properties quoted in the mentioned below Darbo-Sadovskij-type fixed point theorem for condensing admissible maps (it is a slight generalization of \cite[Th.59.12]{gorn}).
\begin{theorem}\label{4}
Let $\gamma$ be a monotone semiadditive semi-homogeneous regular and algebraically semiadditive MNC defined on bounded subsets of the space $E$ and attaining values in some closed cone in a Banach space. Suppose that $C$ is nonempty closed convex and bounded subset of the space $E$. If $F\colon C\map C$ is an admissible $\gamma$-condensing multivalued map, then $F$ has a fixed point.
\end{theorem}
\par If $H(\cdot)$ denotes any continuous (co)homology functor with coefficients in the field of rational numbers $\varmathbb{Q}$ (for instance, the \v Cech homology $\check{H}_*$ or cohomology $\check{H}^*$ with compact carriers), then the space $X$ having the property
\[H_q(X)=
\begin{cases}
0&\mbox{for }q\geqslant 1,\\
\varmathbb{Q}&\mbox{for }q=0
\end{cases}\]
is called acyclic. In other words its homology are exactly the same as the homology of a one point space. A compact (nonempty) space $X$ is an $R_\delta$-set if there is a decreasing sequence of contractible compacta $(X_n)_\n$ containing $X$ as a closed subspace such that $X=\bigcap_\n X_n$ (compare \cite{hyman}). In particular, $R_\delta$-sets are acyclic. 
\par Let us move on to the key issue of the assumptions, on which the results of this paper are based. Assume that $p$ is a real number from the interval $[1,\infty)$. Fix a compact segment $I=[0,T]$ for some end time $T>0$. Let $F\colon I\times E\map E$ be a set-valued map. We will use the following hypotheses on the mapping $F$:
\begin{itemize}
\item[$(F_1)$] for every $t\in I$, $x\in E$ the set $F(t,x)$ is nonempty, closed and convex; when $p=1$, then $F(t,x)$ is additionally a weakly compact set for almost all $t\in I$ and for all $x\in E$,
\item[$(F_2)$] the map $F(\cdot,x)$ has a strongly measurable selection for every $x\in E$,
\item[$(F_3)$] the map $F(t,\cdot)$ is upper hemicontinuous for almost all $t\in I$,
\item[$(F_4)$] there is $c\in L^p(I,\R{})$ such that $||F(t,x)||^+=\sup\{|y|\colon y\in F(t,x)\}\<c(t)(1+|x|)$ for almost all $t\in I$ and for all $x\in E$,
\item[$(F_5)$] there is a function $\eta\in L^p(I,\R{})$ such that for all bounded subsets $\Omega\subset E$ and for almost all $t\in I$ the inequality holds \[\beta(F(t,\Omega))\<\eta(t)\beta(\Omega).\]
\end{itemize}
\par Denote by $\bigtriangleup$ the set $\{(t,s)\in I\times I\colon 0\<s\<t\<T\}$. We shall also assume that the mapping $k\colon\bigtriangleup\to{\mathcal L}(E)$ possesses the following properties:
\begin{itemize}
\item[$(k_1)$] the function $k(\cdot,s)\colon[s,T]\to{\mathcal L}(E)$ is differentiable for every $s\in I$,
\item[$(k_2)$] the function $k(t,\cdot)\colon[0,t]\to{\mathcal L}(E)$ is continuous for all $t\in I$,
\item[$(k_3)$] the function $k(\cdot,\cdot)\colon\{(t,t)\colon t\in I\}\to{\mathcal L}(E)$ is continuous, whereas the operator $k(t,t)$ is invertible for all $t\in I$,
\item[$(k_4)$] there exists a function $\psi\in L^q(I,\R{})$ such that $q^{-1}+p^{-1}=1$ and for every $(t,s)\in\bigtriangleup$ we have $\left\Arrowvert\frac{\partial}{\partial t}k(t,s)\right\Arrowvert_{{\mathcal L}}\<\psi(s)$,
\item[$(k_5)$] for every $t\in I$, $k(t,\cdot)\in L^q([0,t],{\mathcal L}(E))$, where $q^{-1}+p^{-1}=1$,
\item[$(k_6)$] the function $I\ni t\mapsto k(t,\cdot)\in L^q([0,t],{\mathcal L}(E))$ is continuous in the norm $||\cdot||_q$ of the space $L^q(I,{\mathcal L}(E))$,
\item[$(k_7)$] the operator $k(t,s)$ is completely continuous for all $(t,s)\in\bigtriangleup$.
\end{itemize}
\begin{remark}
In view of the mean value theorem it is clear that
\[||k(t,s)||_{{\mathcal L}}\<||k(s,s)||_{{\mathcal L}}+\sup_{\xi\in[s,t]}\left\Arrowvert\frac{\partial}{\partial t}k(\xi,s)\right\Arrowvert_{{\mathcal L}}(t-s)\] for every $(t,s)\in\bigtriangleup$. From this
\[||k(t,s)||_{{\mathcal L}}^q\<2^{q-1}\sup_{z\in I}||k(z,z)||_{{\mathcal L}}^q+2^{q-1}T^q\psi(s)^q,\] by the conditions $(k_3)$-$(k_4)$. Applying Lebesgue dominated convergence theorem we arrive at the conclusion that the set of assumptions $(k_1)$-$(k_4)$ is significantly stronger than conditions $(k_5)$-$(k_6)$.
\end{remark}
\par Recall that the Nemtyskij operator $N_F^p\colon C(I,E)\map L^p(I,E)$, corresponding to $F$, is a multivalued map defined by
\[N_F^p(x)=\{w\in L^p(I,E)\colon w(t)\in F(t,x(t))\mbox{ for almost all }t\in I\}.\]
Denote by $V\colon L^p(I,E)\to C(I,E)$ the following classical Volterra integral operator:
\begin{equation}\label{volterra}
V(w)(t)=\int_0^t k(t,s)w(s)ds,\;t\in I.
\end{equation}
To justify above-mentioned definition or to ensure the continuity of the operator $V$ it is enough to impose conditions $(k_5)$-$(k_6)$ on the kernel $k$. This plain observation rests on the H\"older inequality. \par Considerations of this part of the study are devoted to integral inclusions of the following form:
\begin{equation}\label{inclusion}
x(t)\in h(t)+\int_0^tk(t,s)F(s,x(s))\,ds,\;t\in I,
\end{equation}
where $h\in C(I,E)$ and $p\in[1,\infty)$ are fixed. By a solution of this inclusion we mean a function $x\in C(I,E)$, which satisfies equation
\[x(t)=h(t)+\int_0^tk(t,s)w(s)\,ds,\;t\in I\]
for some $w\in L^p(I,E)$ such that $w(t)\in F(t,x(t))$ for almost all $t\in I$.
\par Preparing the ground for the proof of theorems describing the structure of the set $S^p_{\!F}$ of solutions to integral inclusion \eqref{inclusion} we will justify several auxiliary statements. The first is a lemma summarizing the properties of the Niemytskij operator. It should be noted that the scheme of proof of this fact is known to a large extent and occurs, in the context of the case $p = 1$, in the paper \cite{bothe}.
\begin{proposition}\label{Niemytskij}
If the map $F$ fulfills conditions $(F_1)$-$(F_4)$ and the space $E$ is reflexive for $p\in(1,\infty)$, then the operator $N_F^p$ has nonempty convex weakly compact values and it is a weakly upper semicontinuous multivalued map.
\end{proposition}
\begin{proof} 
For any $x\in C(I,E)$ there is a sequence $(x_n)_\n$ of step functions, converging uniformly to $x$ on $I$. Accorgingly to the assumption $(F_2)$ we can indicate a strongly measurable map $w_n$ such that $w_n(t)\in F(t,x_n(t))$, i.e. $(x_n(t),w_n(t))\in\Graph(F(t,\cdot))$, for almost all $t\in I$. Using condition $(F_4)$ we get \[|w_n(t)|\<||F(t,x_n(t))||\<c(t)(1+\sup_\n||x_n||)\] almost everywhere in $I$. Thus $w_n\in N_F^p(x_n)$. \par Suppose $p=1$. Then $F(t,\cdot)$ is a weakly upper semicontinuous map for almost all $t\in I$. Taking into account that $\{w_n(t)\}_\n\subset F(t,\overline{\{x_n(t)\}}_\n)$ for almost all $t\in I$, we infer that the sequence $(w_n)_\n$ is relatively weakly compact in view of Theorem \ref{2}. If $w$ is a weak limit of some subsequence of $(w_n)_\n$, then by Theorem \ref{convergence}., $w(t)\in F(t,x(t))$ for almost all $t\in I$. Therefore $w\in N_F^1(x)$. \par In case $p\in(1,\infty)$ observe that $(w_n)_\n$ is a bounded sequence in the reflexive space $L^p(I,E)$ (keeping in mind that $E$ is reflexive, apply duality theorem (see \cite[Th.IV.1.1]{uhl})). By Eberlein-\u Smulian theorem $(w_n)_\n$ converges weakly to some $w\in L^p(I,E)$ (with an accuracy of a subsequence). Using convergence theorem we see that $w\in N_F^p(x)$. \par In this way we proved that the Niemytskij operator has nonempty values. Applying similar reasoning one proves weak upper semicontinuity of $N_F^p$.
\end{proof}
A separate issue is the case of a single-valued Niemytskij operator $N_{\!f}^p$, corresponding to a mapping $f\colon I\times E\to E$. With regard to function $f$ we will assume what follows:
\begin{itemize}
\item[$(f_1)$] the map $f(\cdot,x)$ is strongly measurable for every $x\in E$,
\item[$(f_2)$] there is $c\in L^p(I,\R{})$ such that $|f(t,x)|\<c(t)(1+|x|)$ for almost all $t\in I$ and for all $x\in E$,
\item[$(f_3)$] the map $f(t,\cdot)\colon E\to E$ is continuous for almost all $t\in I$,
\item[$(f_4)$] the map $f(t,\cdot)$ is weakly continuous for almost all $t\in I$, i.e. $f(t,\cdot)\colon E\to E_w$ is continuous, where $E_w$ stands for $E$ endowed with the weak topology,
\item[$(f_5)$] there is a function $\eta\in L^p(I,\R{})$ such that for all bounded subsets $\Omega\subset E$ and for almost all $t\in I$ the inequality holds \[\beta(f(t,\Omega))\<\eta(t)\beta(\Omega).\]
\end{itemize}
\begin{proposition}
Assume that conditions $(f_1)$, $(f_2)$, $(f_4)$ are fulfilled and the dual space $E^*$ has the Radon-Nikod\'ym property. Then the Niemytskij operator $N_{\!f}^p$ is weakly continuous. If the mapping $f$ satisfies conditions $(f_1)$-$(f_3)$, then the operator $N_{\!f}^p$ is simply continuous.
\end{proposition}
\begin{proof}
Suppose that the mapping $f(t,\cdot)$ is weakly continuous for almost all $t\in I$ (assumption $(f_4)$) and the dual space $E^*$ has the Radon-Nikod\'ym property. Let $x_n\to x$ in $C(I,E)$. Denote $w_n=N_{\!f}^p(x_n)$ and $w=N_{\!f}^p(x)$. We claim that almost everywhere weakly convergent functional sequence $(w_n)_\n$ is in fact convergent in the weak topology of the space $L^p(I,E)$. Indeed, from duality theorem (\cite[Th.IV.1.1]{uhl}) it follows that for every functional $\xi\in L^p(I,E)^*$ there exists such a function $g\in L^q(I,E^*)$ that \[\langle \xi,w_n\rangle=\int_0^T\langle w_n(t),g(t)\rangle\,dt.\] Knowing that $w_n(t)\rightharpoonup w(t)$ for almost all $t\in I$, we get $\langle w_n(t),g(t)\rangle\xrightarrow[n\to\infty]{}\langle w(t),g(t)\rangle$ almost everywhere in $I$. Observe that $||g(\cdot)||_{E^*}c(\cdot)\in L^1(I,\R{})$, because $p^{-1}+q^{-1}=1$. At the same time \[|\langle w_n(t),g(t)\rangle|\<||g(t)||_{E^*}|w_n(t)|\<||g(t)||_{E^*}c(t)(1+\sup_\n||x_n||)\] for almost all $t\in I$ and for every $\n$. Thus, Lebesgue dominated convergence theorem implies \[\int_0^T\langle w_n(t),g(t)\rangle\,dt\xrightarrow[n\to\infty]{}\int_0^T\langle w(t),g(t)\rangle\,dt.\] Actually we have shown that $\langle \xi,w_n\rangle\xrightarrow[n\to\infty]{}\langle \xi,w\rangle$ for every $\xi\in L^p(I,E)^*$, i.e. $w_n\rightharpoonup w$ in $L^p(I,E)$. In this case operator $N_{\!f}^p$ is weakly continuous. \par If condition $(f_4)$ is satisfied, then there is a convergence $f(t,x_n(t))\to f(t,x(t))$ for almost all $t\in I$. Using integral boundedness of the sequence $(f(\cdot,x_n(\cdot)))_\n$ (condition $(f_2)$) and Lebesgue dominated convergence theorem we see that \[\int_0^T|f(t,x_n(t)-f(t,x(t))|^p\,dt\xrightarrow[n\to\infty]{}0,\] i.e. $N_{\!f}^p(x_n)\to N_{\!f}^p(x)$ as $n\to\infty$.
\end{proof}
\par The set $S_{\!F}^p$ of solutions to integral inclusion under consideration obviously coincides with the set of fixed points of the operator ${\mathcal F}\colon C(I,E)\map C(I,E)$, given by
\begin{equation}\label{calF}
{\mathcal F}(x)=h(x)+V\circ N_F^p(x).
\end{equation}
\begin{lemma}\label{calFlem}
Let $p\in[1,\infty)$, while the space $E$ is reflexive for $p\in(1,\infty)$. Assume that either conditions $(F_1)$-$(F_5)$ and $(k_5)$-$(k_6)$ or $(F_1)$-$(F_4)$ and $(k_5)$-$(k_7)$ are satisfied. Then the operator ${\mathcal F}$ is acyclic and condensing relative to MNC $\nu_L$, for some $L>0$. 
\end{lemma}
\begin{proof}
We claim that ${\mathcal F}$ is upper semicontinuous and has nonempty convex compact values. Precisely, we will show that if $x_n\rightrightarrows x$ and $y_n\in {\mathcal F}(x_n)$, then there is a subsequence $(y_{k_n})_\n$ uniformly convergent to $y\in{\mathcal F}(x)$. From Proposition \ref{Niemytskij}. and the fact that $w\mapsto h+V(w)$ is an affine operator it follows that ${\mathcal F}(x)$ is nonempty convex for every $x\in C(I,E)$. Let $x_n\rightrightarrows x$ and $y_n\in{\mathcal F}(x_n)$. Then $y_n=h+V(w_n)$ for some $w_n\in N_F^p(x_n)$. Since the operator $N_F^p$ is weakly upper semicontinuous, there is a subsequence (again denoted by) $(w_n)_\n$ such that $w_n\rightharpoonup w\in N_F^p(x)$. We have the following estimations
\begin{equation}\label{equicon}
\begin{aligned}
|y_n(t)-y_n(\tau)|&=\left|\int_0^t(k(t,s)-k(\tau,s))w_n(s)ds-\int_t^\tau k(\tau,s)w_n(s)ds\right|\\&\<\int_0^t||k(t,s)-k(\tau,s))||_{{\mathcal L}}|w_n(s)|ds+\int_t^\tau ||k(\tau,s)||_{{\mathcal L}}|w_n(s)|ds\\&\<\int_0^t||k(t,s)-k(\tau,s))||_{{\mathcal L}}c(s)\left(1+\sup_\n||x_n||\right)ds\\&+\int_t^\tau||k(\tau,s)||_{{\mathcal L}}c(s)\left(1+\sup_\n||x_n||\right)ds\\&\<\left(1+\sup_\n||x_n||\right)\left(||k(t,\cdot)-k(\tau,\cdot)||_q||c||_p\!+\sup\limits_{t\in I}||k(t,\cdot)||_q\left(\int_t^\tau c(s)^pds\right)^{\frac{1}{p}}\right).
\end{aligned} 
\end{equation}
Assumption $(k_6)$ and the absolute continuity of the Lebesgue integral implies equicontinuity of the family $\{y_n\}_\n$. From Theorem \ref{3}. and assumption $(F_5)$ it follows:
\begin{equation}\label{analog}
\begin{aligned}
\beta(\{y_n(t)\}_\n)&=\beta\left(\left\{h(t)+\int_0^t k(t,s)w_n(s)ds\right\}_\n\right)\<2\int_0^t\beta(k(t,s)\{w_n(s)\}_{n\geqslant 1})ds\\&\<2\int_0^t||k(t,s)||_{{\mathcal L}}\beta(\{w_n(s)\}_\n)ds\<2\int_0^t||k(t,s)||_{{\mathcal L}}\beta(F(s,\{x_n(s)\}_\n))ds\\&\<2\int_0^t||k(t,s)||_{{\mathcal L}}\eta(s)\beta(\{x_n(s)\}_\n))ds
\end{aligned}
\end{equation}
and therefore $\beta(\{y_n(t)\}_\n)=0$ for every $t\in I$. Applying Arzel\`a theorem we gather that the sequence $(y_n)_\n$ is relatively compact. Let $y_{k_n}\rightrightarrows y$. Recall that in the class of linear operators on normed spaces the norm continuity is equivalent to the weak continuity. So if $w_{k_n}\rightharpoonup w$ in $L^p(I,E)$, then $V(w_{k_n})\rightharpoonup V(w)$. Consequently, $y_{k_n}=h+V(w_{k_n})\rightharpoonup h+V(w)$. At the same time $y_{k_n}\rightharpoonup y$, that is $y\in{\mathcal F}(x)$.\par Now we prove that ${\mathcal F}$ is a condensing operator with respect to the MNC $\nu_L$, defined at the point \eqref{measure}. We specify the mapping $\varphi\colon\R{}\to\R{}_+$ by the formula
\begin{equation}\label{fi}
\varphi(L)=\sup\limits_{t\in I}e^{-Lt}\left(\int_0^t(\eta(s)e^{Ls})^p ds\right)^{\frac{1}{p}}.
\end{equation}
It is easy to show that $\varphi(L)\xrightarrow[L\to+\infty]{}0^+$.
Fix $L>0$ so that 
\begin{equation}\label{L}
\varphi(L)<\left(2\sup_{t\in I}||k(t,\cdot)||_q\right)^{-1}.
\end{equation}
Suppose $\Omega\subset C(I,E)$ is bounded set for which the inequality holds
\begin{equation}\label{nierznu}
\nu_L(\Omega)\<\nu_L({\mathcal F}(\Omega)).
\end{equation}
Assume that $\{v_n\}_{n\geqslant 1}\in\Delta({\mathcal F}(\Omega))$ realises measure of the set ${\mathcal F}(\Omega)$, i.e.
\[\nu_L({\mathcal F}(\Omega))=\left(\sup\limits_{t\in I}e^{-Lt}\beta(\{v_n(t)\}_{n\geqslant 1}),\modulus_C(\{v_n\}_{n\geqslant 1})\right).\]
Then there are functions $u_n\in\Omega$ and $w_n\in N_F^p(u_n)$ such that $v_n=h+V(w_n)$. If $\{y_n\}_\n\in\Delta(\Omega)$ is the set where the maximum in the definition of $\nu_L(\Omega)$ is reached then, in accordance with the assumption \eqref{nierznu}, we have
\begin{equation}\label{equ10}
\begin{cases}
\sup\limits_{t\in I}e^{-Lt}\beta(\{y_n(t)\}_\n)\<\sup\limits_{t\in I}e^{-Lt}\beta(\{v_n(t)\}_\n),&\\\modulus_C(\{y_n\}_\n)\<\modulus_C(\{v_n\}_\n).
\end{cases}
\end{equation}
Estimating analogously to inequality \eqref{analog}, we see that
\begin{equation}\label{20}
\beta(\{v_n(t)\}_\n)\<2\!\!\int_0^t\!\!||k(t,s)||_{{\mathcal L}}\eta(s)\beta(\{u_n(s)\}_\n)ds\<\sup\limits_{t\in I}e^{-Lt}\beta(\{u_n(t)\}_\n)\,2\int_0^t||k(t,s)||_{{\mathcal L}}\eta(s)e^{Ls}ds.
\end{equation}
Hence
\[\begin{split}
\sup\limits_{t\in I}e^{-Lt}\beta(\{v_n(t)\}_\n)&\<\sup\limits_{t\in I}e^{-Lt}\beta(\{u_n(t)\}_\n)2\sup\limits_{t\in I}e^{-Lt}\int_0^t||k(t,s)||_{{\mathcal L}}\eta(s)e^{Ls}ds\\&\<\sup\limits_{t\in I}e^{-Lt}\beta(\{u_n(t)\}_\n)2\sup\limits_{t\in I}||k(t,\cdot)||_q\sup\limits_{t\in I}e^{-Lt}\left(\int_0^t(\eta(s)e^{Ls})^pds\right)^{\frac{1}{p}}\\&=2\sup\limits_{t\in I}||k(t,\cdot)||_q \varphi(L)\sup\limits_{t\in I}e^{-Lt}\beta(\{u_n(t)\}_\n)\\&\<2\sup\limits_{t\in I}||k(t,\cdot)||_q \varphi(L)\sup\limits_{t\in I}e^{-Lt}\beta(\{y_n(t)\}_\n).
\end{split}\]
If we now assume that $\sup_{t\in I}e^{-Lt}\beta(\{y_n(t)\}_\n)>0$ then, in view of \eqref{L}, we have
\[\sup\limits_{t\in I}e^{-Lt}\beta(\{v_n(t)\}_\n)<\sup\limits_{t\in I}e^{-Lt}\beta(\{y_n(t)\}_\n).\]
This together with \eqref{equ10} gives a contradiction. Therefore $\sup_{t\in I}e^{-Lt}\beta(\{y_n(t)\}_\n)=0$. Since the set $\Omega$ is bounded, the image ${\mathcal F}(\Omega)$ must be equicontinuous (similarly to \eqref{equicon}). Thus the subset $\{v_n\}_\n\subset{\mathcal F}(\Omega)$ is also equicontinuous, i.e. $\modulus_C(\{v_n\}_\n)=0$. Again using \eqref{equ10} we get: $\modulus_C(\{y_n\}_\n)=0$. Eventually $\nu_L(\Omega)=(0,0)$, which means that $\Omega$ is relatively compact in $C(I,E)$.\par Observe that condition $(F_5)$ is used in the above argumentation only in the estimates \eqref{analog} and \eqref{20}. Modyfing them using the assumption $(k_7)$ we obtain the inequality
\[\beta(\{v_n(t)\}_\n)=\beta\left(\left\{h(t)+\int_0^tk(t,s)w_n(s)\,ds\right\}_\n\right)\<2\int_0^t\beta(k(t,s)\{w_n(s)\}_\n)ds=0\]
for $t\in I$. In particular, this means that $\sup_{t\in I}e^{-Lt}\beta(\{v_n(t)\}_\n)=0$. From assumption \eqref{equ10} we infer that $\sup_{t\in I}e^{-Lt}\beta(\{y_n(t)\}_\n)=0$, proving thereby that operator ${\mathcal F}$ is condensing relative to MNC $\nu_L$.  
\end{proof}
\begin{corollary}\label{calFsing}
Assume that either conditions $(f_1)$-$(f_3)$, $(f_5)$ and $(k_5)$-$(k_6)$ or conditions $(f_1)$-$(f_3)$ and $(k_5)$-$(k_7)$ are satisfied. Then operator ${\mathcal F}=h+V\circ N_{\!f}^p$ is continuous and $\nu_L$-condensing.
\end{corollary}
\begin{corollary}
Suppose the dual space $E^*$ possesses the Radon-Nikod\'ym property. Assume that either conditions $(f_1)$-$(f_2)$, $(f_4)$-$(f_5)$ and $(k_5)$-$(k_6)$ or conditions $(f_1)$-$(f_2)$, $(f_4)$ and $(k_5)$-$(k_7)$ are satisfied. Then operator ${\mathcal F}=h+V\circ N_{\!f}^p$ is weakly continuous and $\nu_L$-condensing.
\end{corollary}
The source of technical difficulties mentioned in the introduction is the ambiguity of the integral description of solutions $x$ to the inclusion \eqref{inclusion} with the involvement of a selection of the set-valued map $F(\cdot,x(\cdot))$. Imposition of conditions $(k_1)$-$(k_4)$ resolves this problem, providing the injectivity of the Volterra operator $V$. In the following lemma this property was deduced from the Leibniz integral rule, differentiating the Volterra operator's kernel under the sign of a Bochner integral.
\begin{lemma}\label{lem2}
Asssume that kernel $k$ satisfies conditions $(k_1)$-$(k_4)$. Then the Volterra {ope\-ra\-tor} $V\colon L^p(I,E)\to C(I,E)$, given by the formula \eqref{volterra}, is injective.
\end{lemma}
\begin{proof}
Note that the mapping $[0,t]\ni s\mapsto\frac{\partial}{\partial t}k(t,s)\in{\mathcal L}(E)$ may be treated as the poinwise limit of the sequence $\left(\frac{k(t+n^{-1},\cdot)-k(t,\cdot)}{n^{-1}}\right)_\n$ of strongly measurable functions. Thus, conditions $(k_1)$-$(k_2)$ and $(k_4)$ imply $\frac{\partial}{\partial t}k(t,\cdot)\in L^q([0,t],{\mathcal L}(E))$ for all $t\in I$.\par Let $V(w_1)=V(w_2)$ and $w=w_1-w_2$. Clearly $\frac{d}{dt}V(w)(t)=0$ for $t\in I$. In particular
\begin{equation}\label{zero}
\lim_{n\to\infty}\left(\int\limits_0^{t-\frac{1}{n}}\frac{k\left(t-\frac{1}{n},s\right)-k(t,s)}{-\frac{1}{n}}w(s)\,ds+n\int\limits_{t-\frac{1}{n}}^tk(t,s)w(s)\,ds\right)=0
\end{equation}
for every $t\in I$. If \[f_n(s)=\frac{k\left(t-\frac{1}{n},s\right)-k(t,s)}{-\frac{1}{n}}w(s)\chi_{[0,\,t-n^{-1}]}(s),\] then $f_n\in L^1([0,t],E)$ and $f_n(s)\xrightarrow[n\to\infty]{}\frac{\partial}{\partial t}k(t,s)w(s)$ for $s\in[0,t]$. Using the mean value theorem and assumption $(k_4)$ we see that for every $\n$ and for $s\in\left[0,t-\frac{1}{n}\right]$ the following estimate is valid
\[|f_n(s)|\<\left\Arrowvert\frac{k\left(t-\frac{1}{n},s\right)-k(t,s)}{-\frac{1}{n}}\right\Arrowvert_{{\mathcal L}}|w(s)|\<\sup_{\xi\in\left[t-\frac{1}{n},t\right]}\left\Arrowvert\frac{\partial}{\partial t}k\left(\xi,s\right)\right\Arrowvert_{{\mathcal L}}|w(s)|\<\psi(s)|w(s)|.\] For points $s\in\left(t-\frac{1}{n},t\right]$ we have equality $|f_n(s)|=0$, which means that \[|f_n(s)|\<\psi(s)|w(s)|\] for all $\n$ and $s\in[0,t]$. From Lebesgue convergence theorem (\cite[Th.II.2.3]{uhl}) it follows that
\[\lim_{n\to\infty}\int\limits_0^{t-\frac{1}{n}}\frac{k\left(t-\frac{1}{n},s\right)-k(t,s)}{-\frac{1}{n}}w(s)\,ds=\lim_{n\to\infty}\int_0^tf_n(s)=\int_0^t\frac{\partial}{\partial t}k(t,s)w(s)\,ds\] for every $t\in I$. On the other hand, we deal with the following estimates:
\begin{align*}
n\int\limits_{t-\frac{1}{n}}^t|k(t,s)w(s)-k(t,t)w(t)|\,ds&\<n\int\limits_{t-\frac{1}{n}}^t||k(t,s)-k(t,t)||_{{\mathcal L}}|w(s)|\,ds+n\int\limits_{t-\frac{1}{n}}^t||k(t,t)||_{{\mathcal L}}|w(s)-w(t)|\,ds\\&=||k(t,\xi(n))-k(t,t)||_{{\mathcal L}}\;n\int\limits_{t-\frac{1}{n}}^t|w(s)|\,ds+||k(t,t)||_{{\mathcal L}}\;n\int\limits_{t-\frac{1}{n}}^t|w(s)-w(t)|\,ds,
\end{align*}
where $\xi(n)\in\left[t-\frac{1}{n},t\right]$ is some point that exists under the mean value theorem to the integral of the product of a continuous function and a Lebesgue integrable function. The set of Lebesgue points of the Bochner integrable map $I\ni s\mapsto w(s)\in E$, i.e. such points that
\[\lim_{n\to\infty}n\int\limits_{t-\frac{1}{n}}^t|w(s)-w(t)|\,ds=0,\] is a set of full measure in the interval $I$ (\cite[Th.II.2.9]{uhl}). Similarly, for the Lebesgue integrable function $I\ni s\mapsto|w(s)|\in\R{}$ the equality $\lim\limits_{n\to\infty}n\int_{t-\frac{1}{n}}^t|w(s)|\,ds=|w(t)|$ holds for almost all $t\in I$. In view of these properties and the continuity of the mapping $k(t,\cdot)$ we find that
\[\lim_{n\to\infty}n\int\limits_{t-\frac{1}{n}}^tk(t,s)w(s)\,ds=k(t,t)w(t)\] for almost all $t\in I$.\par Applying the equality \eqref{zero}, we see that
\[k(t,t)w(t)+\int_0^t\frac{\partial}{\partial t}k(t,s)w(s)\,ds=0\] for almost all $t\in I$. Inverting operator $k(t,t)$ we obtain the dependence
\[w(t)=-k(t,t)^{-1}\int_0^t\frac{\partial}{\partial t}k(t,s)w(s)\,ds\] and as a consequence the following bound
\[|w(t)|\<||k(t,t)^{-1}||_{{\mathcal L}}\int_0^t\left\Arrowvert\frac{\partial}{\partial t}k(t,s)\right\Arrowvert_{{\mathcal L}}|w(s)|\,ds\] for almost all $t\in I$. Using continuity of the mapping $k(\cdot,\cdot)$ on the diagonal of $I\times I$ and the uniform boundedness principle we infer that $M=\sup_{t\in I}||k(t,t)^{-1}||_{{\mathcal L}}<\infty$. Combining this with the assumption $(k_4)$ we see that
\[|w(t)|\<M\int_0^t\psi(s)|w(s)|\,ds\] for almost all $t\in I$. After applying the generalized version of Gronwall inequality, for the case of Lebesgue integrable functions, it becomes clear that $|w(t)|=0$ for almost all $t\in I$. Thus $w_1=w_2$ in $L^p(I,E)$ and operator $V$ is an injection.
\end{proof}
\par An idea of applying the metric projection in the description of geometrical properties of the solutions set of inclusion \eqref{inclusion} was taken from \cite{umanski}. Important from our point of view are also circumstances in which the metric projection mapping can be regarded as singlevalued. For this reason we quote the following known fact.
\begin{proposition}\label{singleton}
Let $A$ be a nonempty convex and weakly compact subset of a strictly convex normed space $(E,|\cdot|)$. Then the set of elements of the best approximation
\[Pr_A(x)=\{y\in A\colon |x-y|=d_A(x)=\inf\{|x-a|\colon a\in A\}\}\] is a singleton for every $x\in E$.
\end{proposition}
\par In support of the contractibility of the set of approximative solutions to inclusion \eqref{inclusion} a special role plays the uniqueness of the existence of solutions to Volterra integral equation with measurable-locally Lipschitzean integrand function. The problem of existence of solutions to Volterra integral equations in Banach spaces is well recognized and extensively discussed in the previous literature on the subject (compare, for instance \cite{bugaj, reg, pre, szufla, vath}). The thesis of the following lemma also includes the equivalent of ``continuous dependence on initial conditions'' for differential equations.
\begin{lemma}\label{equation}
Let $f\colon I\times E\to E$ be a function satisfying the following conditions:
\begin{itemize}
\item[(i)] $f(\cdot,x)\colon I\to E$ is strongly measurable for every $x\in E$,
\item[(ii)] for each compact subset $K$ of the space $E$ there is a function $\mu_K\in L^p(I,\R{})$ and a real number $\delta>0$ such that \[|f(t,x_1)-f(t,x_2)|\<\mu_K(t)|x_1-x_2|\] for every $t\in I$ and for $x_1,x_2\in D(K,\delta)$,
\item[(iii)] there is a function $c\in L^p(I,\R{})$ such that $|f(t,x)|\<c(t)$ for almost all $t\in I$ and for every $x\in E$.
\end{itemize}
Suppose that $k$ satisfies assumptions $(k_5)$-$(k_6)$ for $p^{-1}+q^{-1}=1$. Then the Volterra integral equation
\begin{equation}\label{rownanie}
x(t)=h(t)+\int_0^tk(t,s)f(s,x(s))ds,\;t\in I
\end{equation}
possesses a unique solution for any $h\in C(I,E)$. Moreover, solutions of the equation \eqref{rownanie} depend continuously on the perturbation $h$. 
\end{lemma}
\begin{proof}
Fix $h\in C(I,E)$. From (ii) it follows that there is a function $\mu=\mu_{h(I)}$ and $\delta>0$ such that the mapping $f(t,\cdot)\colon D(h(I),\delta)\to E$ is $\mu(t)$-Lipschitzean for every $t\in I$. Define operator ${\mathcal F}\colon C([0,\eps],E)\to C([0,\eps],E)$ by the formula:
\[{\mathcal F}(x)(t)=h(t)+\int_0^tk(t,s)f(s,x(s))\,ds,\;t\in[0,\eps],\]
where $\eps>0$ is such that
\begin{equation}\label{delta}
\left(\int_0^\eps c(s)^pds\right)^{\frac{1}{p}}\<\delta\left(\sup\limits_{t\in I}||k(t,\cdot)||_q\right)^{-1}.
\end{equation}
If $D_\eps(h,\delta)$ denotes a ball in the space $C([0,\eps],E)$, i.e. $D_\eps(h,\delta)=\{x\in C([0,\eps],E)\colon\max\limits_{t\in[0,\eps]}|x(t)-h(t)|\<\delta\}$, then it is clear that ${\mathcal F}(D_\eps(h,\delta))\subset D_\eps(h,\delta)$. Observe that function $f$ fulfills assumptions $(f_1)$-$(f_3)$ and $(f_5)$. Strictly speaking, with regard to condition $(f_5)$ we have $\beta(f(t,\Omega(t)))\<\mu(t)\beta(\Omega(t))$ for every $\Omega\subset D(h,\delta)$ and $t\in I$, where $\Omega(t)=\{x(t)\colon x\in\Omega\}$. Therefore the operator ${\mathcal F}\colon D_\eps(h,\delta)\to D_\eps(h,\delta)$ is continuous and $\nu_L$-condensing (Corollary \ref{calFsing}.). By Theorem \ref{4}. there is a fixed point of ${\mathcal F}$ which is a local solution to equation \eqref{rownanie}. This solution is unique as we will see. Suppose there are two functions $x,\,y\in C([0,\eps],E)$ satisfying $x={\mathcal F}(x)$ i $y={\mathcal F}(y)$. Applying the assumption (ii) to the compact set $K=x([0,\eps])\cup y([0,\eps])$ we obtain the estimation:
\begin{align*}
|x(t)-y(t)|&\<\int_0^t||k(t,s)||_{{\mathcal L}}|f(s,x(s))-f(s,y(s))|ds\<\int_0^t||k(t,s)||_{{\mathcal L}}\mu_K(s)|x(s)-y(s)|ds\\&\<\sup\limits_{t\in[0,\eps]}||k(t,\cdot)||_q\left(\int_0^t\mu_K(s)^p|x(s)-y(s)|^pds\right)^{\frac{1}{p}}
\end{align*}
for every $t\in[0,\eps]$. Thus
\[|x(t)-y(t)|^p\<\left(\sup\limits_{t\in[0,\eps]}||k(t,\cdot)||_q\right)^p\int_0^t\mu_K(s)^p|x(s)-y(s)|^pds.\] By Gronwall inequality we infer that $\sup_{t\in[0,\eps]}|x(t)-y(t)|^p=0$, i.e. $x=y$.\par Starting from this point the operator ${\mathcal F}$ is treated as a mapping ${\mathcal F}\colon C(I,E)\to C(I,E)$. Let $\pi\colon I\times C(I,E)\to C(I,E)$ be a function defined by
\[\pi(t,x)(s)=
\begin{cases}
x(s)&\mbox{dla }s\in[0,t],\\
x(t)&\mbox{dla }s\in[t,T].
\end{cases}\]
Performed above reasoning regarding the uniqueness of the local solution is obviously true for each subinterval of the interval $I$. Therefore we have
\begin{equation}\label{rownanie1}
\pi(\tau,x)=\pi(\tau,{\mathcal F}(x))\,\wedge\,\pi(\tau,y)=\pi(\tau,{\mathcal F}(y))\Rightarrow \pi(\tau,x)=\pi(\tau,y).
\end{equation}
Introduce the following designation:
\[J=\{t\in I\colon\exists\,x_t\in C(I,E)\;\;x_t=\pi(t,{\mathcal F}(x_t))\}.\]
Let $x$ be a solution of \eqref{rownanie} on the interval $[0,\eps]$. Then the function $x_\eps\in C(I,E)$ such that $x_\eps=\pi(\eps,x)$ satisfies $x_\eps=\pi(\eps,{\mathcal F}(x_\eps))$. Thus $J$ is nonempty. In fact it is easy to see that for all $t\in J$ we have $[0,t]\subset J$, i.e. $J$ is a subinterval of the segment $I$.\par Take a sequence $(t_n)_\n$, which is monotonically convergent in the set $J$ to the point $t_0=\sup J$. So if $m\<n$, then $t_m\<t_n$ and as a result occurs both $\pi(t_m,x_{t_m})=\pi(t_m,{\mathcal F}(x_{t_m}))$ and $\pi(t_m,x_{t_n})=\pi(t_m,{\mathcal F}(x_{t_n}))$. Thus, in accordance with \eqref{rownanie1}, $\pi(t_m,x_{t_m})=\pi(t_m,x_{t_n})$, i.e. $x_{t_m}(s)=x_{t_n}(s)$ for $s\in[0,t_m]$. The sequence $(x_{t_n}(t_0))_\n$ is fundamental in the space $E$. Indeed, for $m\<n$ we have
\begin{equation}\label{equ3}
\begin{aligned}
&|x_{t_n}(t_0)-x_{t_m}(t_0)|=|x_{t_n}(t_n)-x_{t_m}(t_m)|\\&\<|h(t_n)-h(t_m)|+\int_{t_m}^{t_n}||k(t_n,s)||_{{\mathcal L}}|f(s,x_{t_n}(s))|ds\\&+\int_0^{t_m}||k(t_n,s)-k(t_m,s)||_{{\mathcal L}}|f(s,x_{t_m}(s))|ds\\&\<|h(t_n)-h(t_m)|+\int_{t_m}^{t_n}||k(t_n,s)||_{{\mathcal L}}c(s)ds+\int_0^{t_m}||k(t_n,s)-k(t_m,s)||_{{\mathcal L}}c(s)ds\\&\<|h(t_n)-h(t_m)|+\sup\limits_{t\in I}||k(t,\cdot)||_q\left(\int_{t_m}^{t_n}c(s)^pds\right)^{\frac{1}{p}}+||k(t_n,\cdot)-k(t_m,\cdot)||_q||c||_p.
\end{aligned}
\end{equation}
The right-hand side of the above inequality tends to zero as $n,m\to\infty$. Let $x_{t_0}\colon I\to E$ be a mapping defined as follows:
\[x_{t_0}(t)=
\begin{cases}
\,x_{t_n}(t)&\mbox{for }t\in[0,t_n],\\
\,\lim\limits_{n\to\infty}x_{t_n}(t_0)&\mbox{for }t\in[t_0,T].
\end{cases}\]
The correctness of the definition follows from \eqref{rownanie1}. The continuity of $x_{t_0}$ is simple consequence of the equicontinuity of the family $\{x_{t_n}\}_\n$. It is clear that $x_{t_0}(t)={\mathcal F}(x_{t_0})(t)$ for $t\in[0,t_0)$. Estimating analogously to \eqref{equ3} we get
\[x_{t_0}(t_0)=\lim_{n\to\infty}\left(h(t_n)+\int_0^{t_n}k(t_n,s)f(s,x_{t_n}(s))ds\right)=h(t_0)+\int_0^{t_0}k(t_0,s)f(s,x_{t_0}(s))ds.\] 
Therefore $x_{t_0}=\pi(t_0,{\mathcal F}(x_{t_0}))$ and $t_0\in J$.\par Assume that $t_0<T$. Then there is $\eps>0$ and a mapping $x\in C([t_0,t_0+\eps],E)$ satisfying the following integral equation:
\[x(t)=g(t)+\int_{t_0}^tk(t,s)f(s,x(s))ds,\]
where $g\in C([t_0,T],E)$ is such that
\[g(t)=h(t)+\int_0^{t_0}k(t,s)f(s,x_{t_0}(s))ds.\]
Put
\[x_{t_0+\eps}(t)=
\begin{cases}
x_{t_0}(t)&\mbox{for }t\in[0,t_0],\\
x(t)&\mbox{for }t\in[t_0,t_0+\eps],\\
x(t_0+\eps)&\mbox{for }t\in[t_0+\eps,T].
\end{cases}\]
and observe that $x_{t_0+\eps}\in C(I,E)$. Moreover, the following equality holds
\[x_{t_0+\eps}(t)=h(t)+\int_0^tk(t,s)f(s,x_{t_0+\eps}(s))ds\]
for $t\in[0,t_0+\eps]$. We see that $x_{t_0+\eps}=\pi(t_0+\eps,{\mathcal F}(x_{t_0+\eps}))$ which is contrary to the maximality of $t_0$ in $J$. Consequently $T=t_0$. \par We have shown that the mapping $x_{t_0}$ satisfies $x_{t_0}=\pi(T,{\mathcal F}(x_{t_0}))$. Thus, it is the sought unique solution of the equation \eqref{rownanie}. \par Denote by ${\mathcal T}$ the integral part of the operator ${\mathcal F}$, i.e. ${\mathcal F}=h+{\mathcal T}$. As we already know ${\mathcal T}$ is a continuous $\nu_L$-condensing operator for some $L>0$. The solution $x_n$ to the equation \eqref{rownanie} with the perturbation $h_n\in C(I,E)$ satisfies the operational equation $x_n=h_n+{\mathcal T}x_n$. Suppose that $h_n\rightrightarrows h$ as $n\to\infty$. Our goal is to show that $x_n\rightrightarrows x$, where $x$ is such that $x=h+{\mathcal T}x$. \par Reffering to the equicontinuity of $\{h_n\}_\n$ and to the estimation
\[|x_n(t)-x_n(\tau)|\<|h_n(t)-h_n(\tau)|+||k(t,\cdot)-k(\tau,\cdot)||_q||c||_p\!+\sup\limits_{t\in I}||k(t,\cdot)||_q\left(\int_t^\tau c(s)^pds\right)^{\frac{1}{p}}\]
we can deduce: $\modulus_C(\{x_n\}_\n)=0$. Similarly one can justify that $\modulus_C(\{{\mathcal T}x_n\}_\n)=0$. Given the relative compactness of the set $\{h_n\}_\n$ we can write down:
\begin{align*}
\sup_{t\in I}e^{-Lt}\beta(\{x_n(t)\}_\n)&=\sup_{t\in I}e^{-Lt}\beta(\{h_n(t)+{\mathcal T}(x_n)(t)\}_\n)\\&\<\sup_{t\in I}e^{-Lt}\beta(\{h_n(t)\}_\n)+\sup_{t\in I}e^{-Lt}\beta(\{{\mathcal T}(x_n)(t)\}_\n)\\&=\sup_{t\in I}e^{-Lt}\beta(\{{\mathcal T}(x_n)(t)\}_\n).
\end{align*}
Thus $\nu_L(\{x_n\}_\n)\<\nu_L({\mathcal T}\{x_n\}_\n)$, i.e. $\{x_n\}_\n$ is relatively compact in the space $C(I,E)$.\par Take arbitrary subsequence $(x_{k_n})_\n$ of the sequence $(x_n)_\n$. It contains a subsequence $(x_{m_{k_n}})_\n$, convergent to some $y$ in the space $C(I,E)$. The continuity of the operator ${\mathcal T}$ implies convergence $h_{m_{k_n}}+{\mathcal T}x_{m_{k_n}}\rightrightarrows h+{\mathcal T}y$. Hence $y=h+{\mathcal T}y$. The uniquness of the solution to the equation \eqref{rownanie} means that $x_{m_{k_n}}\rightrightarrows x$. Eventually, the sequence $(x_n)_\n$ tends to the solution $x$.
\end{proof}
\section{Structure of the solution set}
Note that the assumptions $(k_5)$-$(k_6)$ and $(F_4)$ shows that there are a priori bounds on the solutions to \eqref{inclusion}, i.e. if $x\in S^p_F$, then
\[|x(t)|\<||h||+\int_0^t||k(t,s)||_{{\mathcal L}}c(s)(1+|x(s)|)ds\<A+B\left(\int_0^tc(s)^p|x(s)|^pds\right)^{\frac{1}{p}}\]
for every $t\in I$, where $B=\sup_{t\in I}||k(t,\cdot)||_q$ and $A=||h||+B||c||_p$. Thus
\begin{equation}\label{equ6}
||x||\<M=2^{1-p^{-1}}A\exp\left(p^{-1}2^{p-1}B^p||c||^p_p\right).
\end{equation}
\par From now on we assume that the map $F$, instead of condition $(F_4)$, satisfies:
\begin{itemize}
\item[$(F_4')$] $||F(t,x)||\<\mu(t)$ for every $x\in E$ and for almost all $t\in I$, where $\mu\in L^p(I,\R{})$.
\end{itemize}
This assumption does not reduce the generality of the considerations set out in each of the following statements. Indeed, if we denote by $r\colon E\to D(0,M)$ a radial retraction onto the closed ball $D(0,M)$ with radius $M$ given by \eqref{equ6}, then the solution set $S^p_{\!\hat{F}}$ to the integral inclusion
\[x\in h+V\circ N^p_{\!\hat{F}}(x),\]
where the set-valued map $\hat{F}\colon I\times E\map E$ is such that $\hat{F}(t,x)=F(t,r(x))$, coincides with the set $S^p_{\!F}$. Evidently, the map $\hat{F}$ satisfies assumptions $(F_1)$, $(F_2)$ and $(F_4')$ for $\mu(\cdot)=(1+M)c(\cdot)$. Upper hemicontinuity of the map $\hat{F}(t,\cdot)$ follows from the fact that a radial retraction is a Lipschitzean function. The retraction $r$ is also 1-$\beta$-contractive and therefore $\hat{F}$ satisfies assumption $(F_5)$.\par Let us pass on to our prime subject, i.e. to theorems describing the structure of the solution set to the considered integral inclusion. The first of these gives the topological characteristics of this structure.
\begin{theorem}\label{solset}
Let $p\in[1,\infty)$, while the space $E$ is reflexive for $p\in(1,\infty)$. Assume that conditions $(k_5)$-$(k_6)$ and $(F_1)$-$(F_5)$ are fulfilled. Then the set of solutions $S^p_{\!F}$ of the integral inclusion \eqref{inclusion} is nonempty and compact.
\end{theorem}
\begin{proof}
Let ${\mathcal F}$ be a set-valued map given by \eqref{calF}. Observe that ${\mathcal F}$ is bounded. Actually, if $y\in{\mathcal F}(x)$ for some $x\in C(I,E)$, then $||y||\<R=||h||+\sup_{t\in I}||k(t,\cdot)||_q||\mu||_p$, where $\mu$ is the integral bound of $F$ under the assumption $(F_4')$. Due to Lemma \ref{calFlem}. operator ${\mathcal F}\colon D(0,R)\map D(0,R)$ is an acyclic and $\nu_L$-condensing multivalued map. In view of Theorem \ref{4}. it has a fixed point. This fixed point constitutes a solution to inclusion \eqref{inclusion}.\par In order to show the compactness of $S_{\!F}^p$, fix any sequence $(x_n)_\n$ of elements of $S^p_{\!F}$. Assuming that $x_n=h+V(w_n)$ for some $w_n\in N_F^p(x_n)$ we get the relation
\[|x_n(t)-x_n(\tau)|\<|h(t)-h(\tau)|+||k(t,\cdot)-k(\tau,\cdot)||_q||\mu||_p+||k(\tau,\cdot)||_q\left(\int_t^\tau|\mu(s)|^pds\right)^{\frac{1}{p}}\!\!\!.\]
From this we deduce that the family $\{x_n\}_\n$ is equicontinuous. Thus the mapping $I\ni t\mapsto \{x_n(t)\}_\n\subset E$ is continuous with respect to Hausdorff metric. The continuity of the MNC $\beta$ implies the continuity of the function $f\colon I\to\R{}$ such that $f(t)=\beta(\{x_n(t)\}_\n)^p$. Estimating similar to \eqref{analog} we obtain:
\[\beta(\{x_n(t)\}_\n)\<2\int_0^t\!||k(t,s)||_{{\mathcal L}}\eta(s)\beta(\{x_n(s)\}_\n)ds\<2\sup_{t\in I}||k(t,\cdot)||_q\left(\int_0^t\!\eta(s)^p\beta(\{x_n(s)\}_\n)^pds\right)^{\frac{1}{p}}\!\!,\]
that is \[f(t)\<2^p\left(\sup_{t\in I}||k(t,\cdot)||_q\right)^p\int_0^t\eta(s)^pf(s)ds\] for $t\in I$. Applying Gronwall inequality we find that $f$ is a zero function. Therefore Arzel\`a theorem enables us to choose a subsequence $(x_{k_n})_\n$, converging uniformly to some $x$. Proposition \ref{Niemytskij}. implies existence of a subsequence $(w_{m_{k_n}})_\n$, which converges weakly in $L^p(I,E)$ to $w\in N^p_F(x)$. Since $(x_{m_{k_n}})_\n$ tends weakly to $h+V(w)$, so the limit point $x$ must be an element of $S^p_{\!F}$. 
\end{proof}
\begin{corollary}\label{solset2}
Let $p\in[1,\infty)$, while the space $E$ is reflexive for $p\in(1,\infty)$. Assume that conditions $(k_5)$-$(k_7)$ and $(F_1)$-$(F_4')$ are fulfilled. Then the set of solutions $S^p_{\!F}$ of the integral inclusion \eqref{inclusion} is nonempty and compact.
\end{corollary}
\par The main result of this paper, concerning the geometric structure of the set of solutions of inclusion \eqref{inclusion}, contains the following:
\begin{theorem}\label{geometry}
Let $p\in[1,\infty)$, while the space $E$ is reflexive for $p\in(1,\infty)$. Assume that conditions $(k_1)$-$(k_4)$ and $(F_1)$-$(F_5)$ are fulfilled. If any of the following conditions holds:
\begin{itemize}
\item[$(E_1)$] $4\sup\limits_{t\in I}||k(t,\cdot)||_q||\eta||_p<1$, where $p^{-1}+q^{-1}=1$,
\item[$(E_2)$] the space $E$ is strictly convex,
\item[$(E_3)$] the space $E$ is separable, 
\end{itemize}
then the solution set $S^p_{\!F}$ of the integral inclusion \eqref{inclusion} is an $R_\delta$-set in the space $C(I,E)$.
\end{theorem}
\begin{proof}
\par Ad $(E_1)$: We claim that there is a nonempty compact set $X\subset E$ possessing the following property:
\begin{equation}\label{X}
h(t)+\int_0^tk(t,s)\,\overline{\co}F(s,X)\,ds\subset X
\end{equation}
for every $t\in I$. The set $S^p_{\!F}$ is bounded by $M=||h||+\sup_{t\in I}||k(t,\cdot)||_q||\mu||_p$. Let $X_0=D(0,M)$ and $X_n=\overline{\bigcup_{t\in I}Y_n(t)}$, where
\[Y_n=\left\{y\in C(I,E)\colon y(t)\in h(t)+\int_0^tk(t,s)\,\overline{\co}F(s,X_{n-1})\,ds\mbox{ dla }t\in I\right\}.\] It is easy to verify by induction the inclusion $\bigcup_{t\in I}S^p_F(t)\subset X_n$ for $\n$. Observe as well that the family $\{X_n\}_\n$ is decreasing. Put $X=\bigcap_\n X_n$. \par Denote by ${\mathcal B}$ the family of all bounded subsets of the space $E$ and define a function $\w{\beta}\colon{\mathcal B}\to\R{}$ by the formula:
\begin{equation}\label{sequentialbeta}
\w{\beta}(\Omega)=\max\limits_{D\in\Delta(\Omega)}\beta(D),\;\Omega\in{\mathcal B}.
\end{equation}
The mapping $\w{\beta}$ is a measure of noncompactness generated by a sequential Hausdorff MNC and as such inherits properties of MNC $\beta$ (\cite[Th.1.4.2]{sadovski}). In particular, the estimations hold: $\beta(\Omega)\<2\w{\beta}(\Omega)\<2\beta(\Omega)$ for every $\Omega\in{\mathcal B}$ (\cite[Th.1.4.5]{sadovski}).\par Suppose that $\w{\beta}(Y_{n+1}(t))=\beta\left(\left\{h(t)+\int_0^tk(t,s)w_k(s)\,ds\right\}_{k\geqslant 1}\right)$. Then
\begin{equation}\label{equ7}
\begin{aligned}
\beta(Y_{n+1}(t))&\<2\w{\beta}(Y_{n+1}(t))\<4\!\!\int_0^t\!||k(t,s)||_{{\mathcal L}}\beta(\{w_k(s)\}_{k\geqslant 1})\,ds\<4\int_0^t||k(t,s)||_{{\mathcal L}}\beta(\overline{\co}F(s,X_n))\,ds\\&\<4\int_0^t||k(t,s)||_{{\mathcal L}}\eta(s)\beta(X_n)\,ds\<4||k(t,\cdot)||_q||\eta||_p\beta(X_n).
\end{aligned}
\end{equation}
Obviously, the sets $Y_n$ are equicontinuous in $C(I,E)$. So if we fix any sequence $(t_k)_\K$ of points from the compact interval $I$, then the convergence of some subsequence $(t_{m_k})_\K$ to the point $t_0\in I$ implies $d_H(Y_n(t_{m_k}),Y_n(t_0))\to 0$ as $k\to\infty$ (where $d_H$ is a Hausdorff metric). As a result we see that the family $\{Y_n(t)\}_{t\in I}$ is compact in the hyperspace of subsets with a Hausdorff metric. Applying the generalized formula for the measure of sum (\cite[Prop.2]{lesniak}) we get the equality: $\beta\left(\bigcup_{t\in I}Y_n(t)\right)=\sup_{t\in I}\beta(Y_n(t))$. Continuing estimation \eqref{equ7} we obtain
\[\beta(X_{n+1})=\beta\left(\bigcup_{t\in I}Y_{n+1}(t)\right)=\sup_{t\in I}\beta(Y_{n+1}(t))\<2\sup_{t\in I}\w{\beta}(Y_{n+1}(t))\<4\sup_{t\in I}||k(t,\cdot)||_q||\eta||_p\,\beta(X_n).\] The above inequality, together with the assumption $(E_1)$ implies convergence $\beta(X_n)\to 0$ as $n\to\infty$. Thus $\beta(X)=0$, which means that $X$ is a compact set. Since $X\subset X_{n-1}$ for $\n$, so for any $\n$ and $t\in I$ the inclusion $h(t)+\int_0^tk(t,s)\overline{\co}F(s,X)\,ds\subset\bigcup_{t\in I}Y_n(t)\subset X_n$ holds. Thus, the set $X$ has the property \eqref{X}.\par Relying on the compactness of the designed set $X$ we define a mapping $\w{F}\colon I\times E\map E$ by the formula
\[\w{F}(t,x)=\overline{\co}F(t,Pr_X(x)),\] where $Pr_X\colon E\map X$ is a metric projection onto $X$, i.e.
\[Pr_X(x)=\{y\in X\colon |x-y|=d_X(x)\}.\] Observe that so defined multimap satisfies conditions $(F_1)$, $(F_2)$, $(F_4')$ and $(F_5)$ (in fact $\w{F}$ is a compact map). Values of $\w{F}$ are nonempty, because $X$ is proximinal. Using upper semicontinuity and the compactness of values of $Pr_X$ and the condition $(F_3)$ we can match such a radius $\delta>0$ with $\eps>0$ that
\begin{align*}
\sigma(p,\overline{\co}F(t,Pr_X(B(x,\delta))))&=\sigma(p,F(t,Pr_X(B(x,\delta))))\<\sigma(p,F(t,B(Pr_X(x),\gamma)))\<\sigma(p,F(t,\!\!\bigcup_{y\in Pr_X(x)}\!\!\!\!B(y,\gamma_y)))\\&=\!\!\sup_{y\in Pr_X(x)}\!\!\sigma(p,F(t,B(y,\gamma_y)))<\!\!\sup_{y\in Pr_X(x)}\!\!\sigma(p,F(t,y))+\eps=\sigma(p,\overline{\co}F(t,Pr_X(x)))+\eps.
\end{align*}
This means that $\w{F}(t,\cdot)$ is upper hemicontinuous. Replacement of the map $F$ by the map $\w{F}$ does not change the set of solutions to inclusion \eqref{inclusion}, i.e. $S_{\!F}^p=S^p_{\!\w{F}}$.\par Relying on standard approximations (\cite[Lemma 2.2]{deimling}) of the map $\w{F}$ we will show that $S^p_{\!\w{F}}$ is an $R_\delta$-set. Let $r_n=3^{-n}$ for $\n$. Paracompactness of the space $E$ implies the existence of locally Lipschitzean partition of unity $\{\lambda_y^n\colon E\to[0,1]\}_{y\in E}$ such that the family of supports $\{\supp\lambda_y^n\}_{y\in E}$ forms a locally finite (closed) covering of the space $E$ inscribed into the covering $\{B(y,r_n)\}_{y\in E}$, i.e. $\supp\lambda_y^n\subset B(y,r_n)$ for $y\in E$. If $F_n\colon I\times E\map E$ is a multimap defined by the formula
\begin{equation}\label{approx}
F_n(t,x)=\sum_{y\in E}\lambda_y^n(x)\,\overline{\co}\w{F}(t,B(y,2r_n))
\end{equation}
then we have the following string of inclusions
\begin{equation}\label{fn}
\w{F}(t,x)\subset F_{n+1}(t,x)\subset F_n(t,x)\subset\overline{\co}\,\w{F}(t,B(x,3r_n))\subset\overline{\co}F(t,X)
\end{equation}
for every $(t,x)\in I\times E$ and $n\geqslant 1$. Denote by $S_n^p$ the set of solutions to inclusion \eqref{inclusion}, where $F$ is replaced by the mapping $F_n$. From \eqref{fn} it follows that the sets $S_n^p$ form a deacreasing family and that $S^p_{\w{F}}\subset\bigcap_{\n}S_n^p$.
\par Let $x\in\bigcap_{\n}S^p_n$. There is a sequence $(x_n)_\n$ convergent to $x$ uniformly on $I$ such that $x_n\in S^p_n$ for $\n$. If $x_n=h+V(w_n)$, where $w_n\in N_{F_n}^p(x_n)$, then the sequence $(w_n)_\n$ contains a subsequence $(w_{k_n})_\n$ weakly convergent to some function $w$ (use either Theorem \ref{2}. for weakly compact set $\overline{\co}F(t,X)$ or Eberlein-\v Smulian theorem in case $p>1$). Since $w_n(t)\in F_n(t,x_n(t))\subset\overline{\co}\w{F}(t,B(x_n(t),3r_n))$ for almost all $t\in I$, so $w(t)\in\w{F}(t,x(t))$ almost everywhere in $I$, by Theorem \ref{convergence}. (note that mappings $F_n(t,\cdot)$ are actually locally Lipschitzean for sufficiently large n). On the other hand, from $x_n=h+V(w_n)\rightharpoonup h+V(w)$ follows that $x=h+V(w)$, i.e. $x\in S^p_{\w{F}}$. Similarly we can show that the sets $S_n^p$ are closed.
\par Fix $\n$ and take a sequence $(x_k)_\K$ of elements of the set $S_n^p$, i.e. $x_k=h+V(w_k)$, where $w_k\in N_{F_n}^p(x_k)$. Taking into account that $|w_k(t)|\<||F_n(t,x_k(t))||^+\<\mu(t)$ for $\K$ and for almost all $t\in I$, we obtain
\[|x_k(t)-x_k(\tau)|\<|h(t)-h(\tau)|+||k(t,\cdot)-k(\tau,\cdot)||_q||\mu||_p+||k(\tau,\cdot)||_q\left(\int_t^\tau\mu(s)^pds\right)^{\frac{1}{p}}\]
for every $\K$. Consequently, the family $\{x_k\}_\K$ is equicontinuous. For every $t\in I$ the true is that
\[\{x_k(t)\}_\K=\left\{h(t)+\int_0^tk(t,s)w_k(s)\,ds\right\}_\K\subset h(t)+\int_0^tk(t,s)\overline{\co}F(s,X)\,ds\subset X.\] Therefore the set $S_n^p$ is compact in $C(I,E)$.\par For every $\n$ we define a function $f_n\colon I\times E\to E$ in the following way:
\begin{equation}\label{selection}
f_n(t,x)=\sum_{y\in E}\lambda_y^n(x)g_y(t)\in F_n(t,x),
\end{equation}
where $g_y$ is a strongly measurable selection of $\w{F}(\cdot,y)$, existing in view of $(F_2)$. Locally finite nature of the summation in the definition \eqref{selection} together with the locally Lipschitzeanity of the partition of unity $\{\lambda_y^n\colon E\to[0,1]\}_{y\in E}$ means that for any compact subset $K\subset E$ there are constants $\gamma>0$ and $\delta>0$ such that $|f_n(t,x_1)-f_n(t,x_2)|\<\gamma\mu(t)|x_1-x_2|$  for all $t\in I$ and $x_1, x_2\in D (K,\delta)$. Obviously, the mappings $f_n$ remain integrably bounded by $\mu$, while $f_n(\cdot,x)$ are strongly measurable. Thus the thesis of Lemma \ref{equation}. applies to the functions $f_n$.\par Assign to any solution $y\in S^p_n$ a function $w_y\in L^p(I,E)$ such that $w_y\in N_{F_n}^p(y)$ and $y=h+V(w_y)$. Lemma \ref{lem2}. justifies the uniqueness of this choice. Consider continuous mappings having the following form
\[I\ni t\mapsto h(t)+\int_0^{sT}\!\!k(t,\tau)w_y(\tau)\,d\tau,\]
where the points $s\in[0,1]$ and $y\in S^p_n$ are fixed. In view of Lemma \ref{equation}. the following integral equation
\[x(t)=h(t)+\int_0^{sT}\!\!k(t,\tau)w_y(\tau)\,d\tau+\int_{sT}^tk(t,\tau)f_n(\tau,x(\tau))\,d\tau\]
posessess a unique solution on the interval $[sT,T]$. Let us denote this solution by $x[s;y]$.\par We claim that sets $S^p_n$ are contractible. Define a homotopy $H_n\colon[0,1]\times S^p_n\to S^p_n$ by the formula:
\begin{equation}\label{homotopia}
H_n(s,y)(t)=
\begin{cases}
y(t)&\mbox{for }t\in[0,sT],\\
x[s;y](t)&\mbox{for }t\in[sT,T].
\end{cases} 
\end{equation}\setcounter{Page}{\value{page}}
The reception of such a definition means that $H_n(1,y)=x[1;y]=y$. On the other hand $H_n(0,y)=x[0;y]=x[0]$ is a unique solution of the equation
\[x(t)=h(t)+\int_0^tk(t,s)f_n(s,x(s))ds,\;t\in I.\]
Since $x[0]\in h+V\circ N_{f_n}^p(x[0])\subset h+V\circ N_{F_n}^p(x[0])$, so the point $x[0]\in S^p_n$. We will show that ``the contraction'' of the set $S^p_n$ to the point $x[0]$ has a continuous character.
\par Let $(y_k)_\K$ be convergent to $y_0$ in the space $S^p_n$ endowed with the uniform convergence metric, while $s_k\to s_0$ in the segment $[0,1]$. Limiting our considerations to the case $s_kT\nearrow s_0T$, we see that functions $x_k=x[s_k;y_k]$ stand for solutions of integral equations:
\[x_k(t)=h(t)+\int_0^{s_kT}\!\!k(t,\tau)w_{y_k}(\tau)\,d\tau+\int_{s_kT}^tk(t,\tau)f_n(\tau,x_k(\tau))d\tau,\;t\in[s_kT,T].\] Let us define functions $g_k\colon I\to E$ by the formula:
\[g_k(t)=h(t)+\int_0^{s_kT}\!\!k(t,\tau)w_{y_k}(\tau)\,d\tau.\] Recall at this point that $k(t,s)=0$ for $t<s$. Let $V_{s_0T}\colon L^p(I,E)\to C(I,E)$ be the following integral operator:
\[V_{s_0T}(w)(t)=\int_0^{s_0T}\!\!k(t,\tau)w(\tau)\,d\tau,\;t\in I.\] At that time $V_{s_0T}(w_{y_k})\rightrightarrows V_{s_0T}(w_{y_0})$, with a precision of a subsequence. Indeed,
\begin{equation}\label{V_sT}
\begin{aligned}
\beta(\{V_{s_0T}(w_{y_k})(t)\}_\K)\<2\!\!\int_0^{s_0T}\!\!\!\!\!\!||k(t,\tau)||_{{\mathcal L}}\beta(\overline{\co}F(\tau,X))\,d\tau\<2\beta(X)\sup_{t\in I}||k(t,\cdot)||_q\left(\int_0^{s_0T}\eta(\tau)^pd\tau\right)^\frac{1}{p},
\end{aligned}
\end{equation}
i.e. $\sup_{t\in I}\beta(\{V_{s_0T}(w_{y_k})(t)\}_\K)=0$. Equicontinuity of $\{V_{s_0T}(w_{y_k})\}_\K$ is a straight consequence of the estimation
\[|V_{s_0T}(w_{y_k})(t)-V_{s_0T}(w_{y_k})(\tau)|\<\int_0^{s_0T}\!\!||k(t,z)-k(\tau,z)||_{{\mathcal L}}|w_{y_k}(z)|\,dz\<||k(t,\cdot)-k(\tau,\cdot)||_q||\mu||_p.\] Therefore $(V_{s_0T}(w_{y_k}))_\K$ contains uniformly convergent subsequence $(V_{s_0T}(w_{y_{m_k}}))_\K$. Concurrently, there is a subsequence (again denoted by) $(w_{y_{m_k}})_\K$ such that $w_{y_{m_k}}\rightharpoonup w$ in $L^p(I,E)$. The weak continuity of operator $V$ implies: $h+V(w_{y_{m_k}})\rightharpoonup h+V(w)$. Since $y_k\rightrightarrows y_0=h+V(w_{y_0})$, we see that $V(w)=V(w_{y_0})$. In view of Lemma \ref{lem2}. we infer that $w_{y_{m_k}}\rightharpoonup w_{y_0}$ in $L^p(I,E)$. Of course, $V_{s_0T}$ is also linear and continuous, so $V_{s_0T}(w_{y_{m_k}})\rightharpoonup V_{s_0T}(w_{y_0})$. Finally: $V_{s_0T}(w_{y_{m_k}})\rightrightarrows V_{s_0T}(w_{y_0})$.\par Let us represent functions $g_k$ in the following form:
\[g_k(t)=h(t)+V_{s_0T}(w_{y_k})-\int_{s_kT}^{s_0T}\!\!k(t,\tau)w_{y_k}(\tau)\,d\tau.\] Given that
\[\sup_{t\in I}\left|\int_{s_kT}^{s_0T}\!\!k(t,\tau)w_{y_k}(\tau)\,d\tau\right|\<\sup_{t\in I}||k(t,\cdot)||_q\left(\int_{s_kT}^{s_0T}\mu(\tau)^p\right)^\frac{1}{p}\xrightarrow[k\to\infty]{}0,\] there is a subsequence of $(g_k)_\K$ uniformly convergent on the interval $I$ to the mapping $g\colon I\to E$ such that \[g(t)=h(t)+\int_0^{s_0T}\!\!k(t,\tau)w_{y_0}(\tau)\,d\tau.\] We have the following estimations:
\[\sup_{t\in[0,s_kT]}|H_n(s_k,y_k)(t)-H_n(s_0,y_0)(t)|\<\sup_{t\in[0,s_kT]}|y_k(t)-y_0(t)|\<||y_k-y_0||\xrightarrow[k\to\infty]{}0,\]
\begin{align*}
\sup_{t\in[s_kT,s_0T]}\!\!\!\!\!|H_n(s_k,y_k)(t)-H_n(s_0,y_0)(t)|\<||g_k-g||+2\sup_{t\in I}||k(t,\cdot)||_q\left(\int_{s_kT}^{s_0T}\mu(\tau)^pd\tau\right)^\frac{1}{p}\xrightarrow[k\to\infty]{}0
\end{align*} and
\[\sup_{t\in[s_0T,T]}\!\left|g_k(t)+\!\!\int_{s_kT}^{s_0T}\!\!\!\!\!\!k(t,\tau)f_n(\tau,x_k(\tau))\,d\tau-g(t)\right|\<||g_k-g||+\sup_{t\in I}||k(t,\cdot)||_q\left(\int_{s_kT}^{s_0T}\mu(\tau)^pd\tau\right)^\frac{1}{p}\underset{k\to\infty}{\longrightarrow}0.\] On the basis of the last inequality and the thesis of Lemma \ref{equation}. (in the part concerning continuous dependence of solutions to integral equation on initial conditions) we find that solutions $x_k$ of the form
\[x_k(t)=g_k(t)+\int_{s_kT}^{s_0T}k(t,\tau)f_n(\tau,x_k(\tau))\,d\tau+\int_{s_0T}^tk(t,\tau)f_n(\tau,x_k(\tau))\,d\tau\] tend uniformly on the interval $[s_0T,T]$ to the function $x[s_0;y_0]$, which is the solution of equation
\[x(t)=g(t)+\int_{s_0T}^tk(t,\tau)f_n(\tau,x(\tau))\,d\tau.\] According to the definition \eqref{homotopia} this convergence means that $\sup_{t\in[s_0T,T]}|H_n(s_k,y_k)(t)-H_n(s_0,y_0)(t)|\to 0$ as $k\to\infty$. Consequently, $||H_n(s_k,y_k)-H_n(s_0,y_0)||\to 0$ as $k\to\infty$, i.e. $H_n$ is a continuous map.\par Summing up the set $S_{\!\w{F}}^p$ is an intersection of a decreasing sequence of contractible compacta, i.e. $S_{\!F}^p$ is an  $R_\delta$-set. 
\par Ad $(E_2)$: We claim that there is a nonempty compact convex set $X\subset C(I,E)$ possessing the following property:
\begin{equation}\label{X2}
h(t)+\int_0^tk(t,s)\,\overline{\co}F(s,X(s))\,ds\subset X(t)
\end{equation}
for every $t\in I$. Let $X_0=D(0,M)$ and $X_n=\overline{Y_n}$, where $M$ is such that $||S_{\!F}^p||^+\<M$. Put
\begin{equation}\label{Yn}
Y_n=\left\{y\in C(I,E)\colon y(t)\in h(t)+\int_0^tk(t,s)\,\overline{\co}F(s,\overline{X_{n-1}(s)})\,ds\mbox{ dla }t\in I\right\}.
\end{equation}
Using the induction step it is easy to justify the inclusion $S_F^p\subset\bigcap_\n X_n$, as well as the convexity of sets $X_n$, which we owe to the envelope $\overline{\co}F(\cdot,\overline{X_{n-1}(\cdot)})$. Thanks to the equicontinuity of $Y_n$ we can apply Michael's theorem to the mapping $I\ni t\mapsto\overline{X_n(t)}\subset E$ and obtain a nonempty set of integrable selections of the multimap $t\mapsto\overline{\co}F(t,\overline{X_n(t)})$. This proves, of course, the correctness of definition \eqref{Yn}.\par Bearing in mind that $\lim_{L\to\infty}\varphi(L)=0$, where $\varphi$ is a mapping given by \eqref{fi}, we choose $L>0$ such that the inequality holds:
\begin{equation}\label{varphi}
\varphi(L)<\left(4\sup_{t\in I}||k(t,\cdot)||_q\right)^{-1}.
\end{equation}
Suppose that
\[\nu_L(Y_n)=\left(\sup_{t\in I}e^{-Lt}\beta(\{v_k(t)\}_\K),\modulus_C(\{v_k\}_\K)\right),\]
where $\nu_L$ is an MNC on the space $C(I,E)$, defined by \eqref{measure}. Then
\begin{align*}
\beta(\{v_k(t)\}_\K)&\<2\int_0^t||k(t,s)||_{{\mathcal L}}\beta(\{w_k(s)\}_\K)\,ds\<2\int_0^t||k(t,s)||_{{\mathcal L}}\beta(\overline{\co}F(s,\overline{X_{n-1}(s)}))\,ds\\&\<2\int_0^t||k(t,s)||_{{\mathcal L}}\eta(s)\beta(X_{n-1}(s))\,ds\<2||k(t,\cdot)||_q\left(\int_0^t(\eta(s)e^{Ls})^p\,ds\right)^\frac{1}{p}\sup_{t\in I}e^{-Lt}\beta(X_{n-1}(t)),
\end{align*}
that is
\begin{align*}
\sup_{t\in I}e^{-Lt}\beta(\{v_k(t)\}_\K)&\<2\sup_{t\in I}||k(t,\cdot)||_q\varphi(L)\sup_{t\in I}e^{-Lt}\beta(X_{n-1}(t))\<4\sup_{t\in I}||k(t,\cdot)||_q\varphi(L)\sup_{t\in I}e^{-Lt}\w{\beta}(X_{n-1}(t))\\&\<4\sup_{t\in I}||k(t,\cdot)||_q\varphi(L)\max_{D\in\Delta(X_{n-1})}\sup_{t\in I}e^{-Lt}\beta(D(t)).
\end{align*}
Taking into account also the second, trivial inequality
\[\modulus_C(\{v_k\}_\K))\<4\sup_{t\in I}||k(t,\cdot)||_q\varphi(L)\max_{D\in\Delta(X_{n-1})}\modulus_C(D)\]
we gain the estimation
\[\nu_L(X_n)=\nu_L(Y_n)\<4\sup_{t\in I}||k(t,\cdot)||_q\varphi(L)\nu_L(X_{n-1}).\]
Using assumption \eqref{varphi} we see that $\nu_L(X_n)\to 0$ as $n\to\infty$. Accordingly $\nu_L(\bigcap_\n X_n)=0$, so the intersection $\bigcap_\n X_n$ is a compact set. Let us put $X=\bigcap_\n X_n$. Now property \eqref{X2} easily follows form the fact that
\[\left\{y\in C(I,E)\colon y(t)\in h(t)+\int_0^tk(t,s)\overline{\co}F(s,X(s))\,ds\mbox{ dla }t\in I\right\}\subset Y_n\] for every $\n$. \par Let $Pr\colon I\times E\map E$ be ``the metric projection with a parameter'', defined in the following way:
\begin{equation}
Pr(t,x)=\{y\in X(t)\colon |x-y|=d_{X(t)}(x)\}.
\end{equation}
See \cite{umanski} to convince oneself that $Pr$ is upper semicontinuous. Moreover, from Proposition \ref{singleton}. it follows that the projection $Pr$ is a singlevalued map. As a result the function $Pr\colon I\times E\to $E must be continuous. In this case the map $\w{F}\colon I\times E\map E$ should be defined as follows:
\begin{equation}\label{tyldaF}
\w{F}(t,x)=\overline{\co}F(t,Pr(t,x)).
\end{equation}
The rest of the proof only slightly differs from the reasoning of the previous point. First, we note that $\w{F}$ satisfies assumptions $(F_1)$-$(F_3)$, $(F_4')$ and $(F_5)$. Applying property \eqref{X2} we show that $S_F^p=S_{\w{F}}^p$. The set of solutions $S_{\!\w{F}}^p$ can be represented in the form of a countable intersection of a deacreasing sequence of compact solution sets $S^p_n$ to integral inclusions corresponding to multivalued approximations $F_n$ such that $F_n(t,x)\subset\overline{\co}F(t,X(t))$ for $(t,x)\in I\times E$. It should be emphasized that the construction of mappings $f_n\colon I\times E\to E$, defined by \eqref{selection}, is based on the possibility of a choice of a strongly measurable selection of the map $F(\cdot,Pr(\cdot,x))$ (remember Proposition \ref{Niemytskij}.). Resting on the uniqueness of solutions to integral equations \[x(t)=g(t)+\int_{sT}^tk(t,\tau)f_n(\tau,x(\tau))\,d\tau,\] we prove the contractibility of sets $S_n^p$. Summarizing we assert that $S_{\!F}^p$ is an $R_\delta$-set.
\par Ad $(E_3)$: Let $\w{F}$ be the set-valued map defined by \eqref{tyldaF}. The separability assumption on the space $E$ does not guarantee the uniqueness of the best approximation realised by the projection $Pr$, but it does ensure the existence of strongly measurable selections of the multivalued map $\w{F}(\cdot,x)$ for every $x\in E$. Indeed, Kuratowski-Ryll Nardzewski theorem (\cite[Th.19.7]{gorn}) implies that the upper semicontinuous, and therefore measurable, map $Pr(\cdot,x)$ possesses a measurable selection $f\colon I\to E$. In view of Pettis measurability theorem (\cite[Th.3.1.3]{talagrand}) we can choose a sequence of simple functions $(f_n)_\n$ converging to $f$ almost everywhere on $I$. If we suppose that $f_n(t)=\sum_{i=1}^{k_n}a^n_i\chi_{A_i^n}(t)$, then there are functions $w_i^n\in L^p(I,E)$ such that $w_i^n(t)\in F(t,a_i^n)$ for almost all $t\in I$ (consequence of assumptions $(F_2)$ and $(F_4')$). Let $w_n=\sum_{i=1}^{k_n}w_i^n\chi_{A_i^n}$. The function $w_n$ is also Bochner integrable and $w_n(t)\in F(t,f_n(t))$ for almost all $t\in I$. If $p=1$, we see that for almost all $t\in I$ the set $\{w_n(t)\}_\n$ is contained in the weakly compact set $F(t,\overline{\{f_n(t)\}_\n})$. Thus the sequence $(w_n)_\n$ is relatively weakly compact in $L^1(I,E)$, thanks to Theorem \ref{1}. Yet the Eberlein-\v Smulian theorem implies the relative weak compactness of the sequence $(w_n)_\n$ for $p\in(1,\infty)$. Denote by $w$ the weak limit of some subsequence $(w_{k_n})_\n$. Applying Theorem \ref{convergence}. we infer that $w(t)\in F(t,f(t))$ almost everywhere in $I$. Thus $w(t)\in F(t,Pr(t,x))$ for almost all $t\in I$, i.e. $w$ is a sought strongly measurable selection of the map $\w{F}(\cdot,x)$. The rest of the proof involves copying of arguments quoted to justify the points $(E_1)$ and $(E_2)$.
\end{proof}
\begin{corollary}\label{geometry2}
Let $p\in[1,\infty)$, while the space $E$ is reflexive for $p\in(1,\infty)$. Assume that conditions $(F_1)$-$(F_4')$ and $(k_1)$-$(k_4)$ together with $(k_7)$ are fulfilled. Then the set of solutions $S^p_{\!F}$ of the integral inclusion \eqref{inclusion} is an $R_\delta$-set in the space $C(I,E)$.
\end{corollary}
\begin{proof}
Scheme of proof is analogous to the justification of the point (E1). Recalling the assumption $(k_7)$ in place of the hypothesis $(F_5)$ estimate \eqref{equ7} and \eqref{V_sT} will gain the following form (taking into consideration that $\{w_k(t)\}_\K$ and $\{w_{y_k}(t)\}_\K$ are bounded in $E$ for almost all $t\in I$)
\[\beta(Y_{n+1}(t))\<2\w{\beta}(Y_{n+1}(t))=2\beta\left(\left\{h(t)+\int_0^t\!k(t,s)w_k(s)\,ds\right\}_{k\geqslant 1}\right)\<4\!\!\int_0^t\!\beta(k(t,s)\{w_k(s)\}_{k\geqslant 1})\,ds=0\] and
\[\beta(\{V_{s_0T}(w_{y_k})(t)\}_\K)=\beta\left(\left\{\int_0^{s_0T}k(t,\tau)w_{y_k}(\tau)\,d\tau\right\}_\K\right)\<2\int_0^{s_0T}\beta(k(t,\tau)\{w_{y_k}(\tau)\}_{k\geqslant 1})\,d\tau=0.\]
\end{proof}\?
\par The fact that a sufficient condition for the existence of fixed points of the operator of translation along the trajectories is acyclicity of his components speaks for weakening of the assumptions upon which the proof of Theorem \ref{geometry}. rests on. It managed to oust the hypothesis $(E_1)$, $(E_2)$ and $(E_3)$ in the next theorem, using the results published in \cite{gelman}.
\begin{theorem}\label{acyclic}
Let $p\in[1,\infty)$, while the space $E$ is reflexive for $p\in(1,\infty)$. Assume that conditions $(F_1)$-$(F_5)$ and $(k_1)$-$(k_4)$ are fulfilled. Then the set of solutions $S^p_{\!F}$ of the integral inclusion \eqref{inclusion} is acyclic.
\end{theorem}
\begin{proof}
Let $(F_n\colon I\times E\map E)_\n$ be a sequence of set-valued approximations of the map $F$, which are defined in analogous manner to \eqref{approx}. As before, set $S_n^p$ denotes the set of solutions to inclusion \eqref{inclusion}, where $F$ is replaced by the mapping $F_n$.
\par It is easy to see that $\bigcap_{\n}\overline{S^p_n}=S^p_F$. Inclusion $S^p_F\subset\bigcap_{\n}\overline{S^p_n}$ is straightforward. Take $x\in\bigcap_{\n}\overline{S^p_n}$. Then there is a sequence $(x_n)_\n$ convergent to $x$ uniformly on $I$ such that $x_n\in S^p_n$ for $\n$. Clearly, we have a $w_n\in N_{F_n}^p(x_n)$, for which $x_n=h+V(w_n)$. Using \eqref{fn} we obtain
\[\{w_n(t)\}_{n=N}^\infty\subset\bigcup_{n=N}^\infty F_N(t,x_n(t))\subset\overline{\co}F(t,\bigcup_{n=N}^\infty B(x_n(t),3r_N))=\overline{\co}F(t,B(\{x_n(t)\}_{n=N}^\infty,3r_N))\]
for every $N\geqslant 1$. Assumption $(F_5)$ and properties of the Hausdorff MNC imply the estimation:
\[\beta(\{w_n(t)\}_\n)=\beta(\{w_n(t)\}_{n=N}^\infty)\<\beta(\overline{\co}F(t,B(\{x_n(t)\}_{n=N}^\infty,3r_N)))\<\eta(t)\beta(B(\{x_n(t)\}_{n=N}^\infty,3r_N))\<3\eta(t)r_N\]
for every $N\geqslant 1$. Since $\eta\in L^p(I,\R{})$, so $\eta(t)<+\infty$ for almost all $t\in I$. Consequently, $\beta(\{w_n(t)\}_\n)=0$ for almost all $t\in I$. Applying \"Ulger's criterion (Theorem \ref{2}.) for the case $p=1$ and Eberlein-\v{S}mulian theorem in all other cases we get a subsequence $(w_{k_n})_\n$ of $(w_n)_\n$, weakly convergent to some function $w$. In view of Theorem \ref{convergence}. it becomes evident that $w(t)\in F(t,x(t))$ for almost all $t\in I$. At the same time we have $x_n=h+V(w_n)\rightharpoonup h+V(w)$. Thus $x=h+V(w)$, i.e. $x\in S^p_F$. \par Let $(x_k)_\K$ be an arbitrary sequence of elemnts of the set $S_n^p$, i.e. $x_k=h+V(w_k)$, where $w_k\in N_{F_n}^p(x_k)$. Denote by $f\colon I\to\R{}$ a scalar function such that $f(t)=\beta(\{x_k(t)\}_\K)$. For every $t\in I$ we have the following inequalities:
\begin{align*}
f(t)&\<2\int\limits_0^t||k(t,s)||_{{\mathcal L}}\beta(\{w_k(s)\}_\K)ds\<2\int\limits_0^t\!||k(t,s)||_{{\mathcal L}}\,\beta(\overline{\co}F(s,B(\{x_k(s)\}_\K,3r_n)))ds\\&\<2\int\limits_0^t||k(t,s)||_{{\mathcal L}}\eta(s)(\beta(\{x_k(s)\}_\K)+3r_n)ds\<2\sup_{t\in I}||k(t,\cdot)||_q\left(\int\limits_0^t\eta(s)^p(f(s)+3r_n)^pds\right)^{\frac{1}{p}},
\end{align*}
so for $t\in I$
\[f(t)^p\<2^pB^p\int\limits_0^t\eta(s)^p(f(s)+3r_n)^pds,\]
where $B=\sup_{t\in I}||k(t,\cdot)||_q$. If the right side of the above inequality we treat as a function $g$ of the variable $t$, then:
\[g'(t)=2^pB^p\eta(t)^p(f(t)+3r_n)^p\]
for almost all $t\in I$. Using the convexity of the mapping $\R{}_+\ni x\to x^p\in\R{}_+$ we get
\[g'(t)\<2^pB^p\eta(t)^p(2^{p-1}f(t)^p+2^{p-1}3^pr_n^p)\<2^{2p-1}B^p\eta(t)^p(g(t)+3^pr_n^p)\] for almost all $t\in I$. Given the fact that the (Carath\'eodory) solution of the following initial value problem
\[\begin{cases}
\dot{x}(t)=2^{2p-1}B^p\eta(t)^p3^pr_n^p+2^{2p-1}B^p\eta(t)^px(t)&\mbox{for almost all }t\in I,\\
x(0)=g(0)=0
\end{cases}\]
majorizes function $g$, we obtain the estimation:
\[f(t)^p\<g(t)\<e^{\,\int\limits_0^t2^{2p-1}B^p\eta(s)^pds}\int\limits_0^t2^{2p-1}B^p3^pr_n^p\eta(s)^pe^{\,-\int\limits_0^s2^{2p-1}B^p\eta(\tau)^pd\tau}ds\]
for every $t\in I$. Consequently
\[f(t)\<3r_n\left(e^{2^{2p-1}B^p\int\limits_0^t\eta(s)^pds}-1\right)^{\frac{1}{p}}\]
for $t\in I$, i.e.
\[\sup\limits_{t\in I}\beta(\{x_k(t)\}_\K)\<3r_n\left(e^{2^{2p-1}B^p||\eta||_p^p}-1\right)^{\frac{1}{p}}.\]
Since the sequence $(x_k)_\K$ was choosen arbitrarily, the last inequality means that:
\begin{equation}\label{max}
\max\limits_{D\in\Delta(S_n^p)}\sup\limits_{t\in I}\beta(D(t))\<Mr_n,
\end{equation}
where $M=3(e^{2^{2p-1}B^p||\eta||_p^p}-1)^{p^{-1}}<+\infty$.
\par We claim that
\begin{equation}\label{beta}
\underset{\eps>0}{\forall}\;\underset{N\in\varmathbb{N}}{\exists}\;\underset{n>N}{\forall}\;S_n^p\subset B(S^p_{\!F},\eps).
\end{equation}
If this statement was not true, then there exists $\eps>0$ and a sequence $(x_{k_n})$ such that $x_{k_n}\in S_{k_n}^p$ and $x_{k_n}\notin B(S^p_{\!F},\eps)$ for every $\n$. Observe that $\{x_{k_n}\}_{n=N}^\infty\in\Delta(S_{k_N}^p)$ for every $N\geqslant 1$. Thus
\[\sup\limits_{t\in I}\beta(\{x_{k_n}(t)\}_{n=1}^\infty)=\sup\limits_{t\in I}\beta(\{x_{k_n}(t)\}_{n=N}^\infty)\<\max\limits_{D\in\Delta(S_{k_N}^p)}\sup\limits_{t\in I}\beta(D(t))\<Mr_{k_N}\]
for every $N\geqslant 1$, by \eqref{max}. Therefore, keeping in mind that $r_n\to 0$ as $n\to\infty$, we obtain the equality: $\sup_{t\in I}\beta(\{x_{k_n}(t)\}_\n)=0$. On the other hand the set $S_n^p$ is equicontinuous, so $\modulus_C(\{x_{k_n}\}_\n)=0$. Arzel\`a theorem implies the existence of a subsequence (again denoted by) $(x_{k_n})$ convergent in $C(I,E)$ to some $x$. Obviously, $x\in \overline{S^p_{k_n}}$ for every $\n$, because the family $\{S_n^p\}_\n$ is decreasing. De facto $x\in\bigcap_\n\overline{S^p_n}$. At the same time $x\notin B(S^p_{\!F},\eps)$. But we have shown that $S^p_{\!F}=\bigcap_\n\overline{S^p_n}$ - contradiction.\par Let $(f_n\colon I\times E\to E)_\n$ be a sequence of functions defined by \eqref{selection}, on the basis of strongly measurable selections of the map $F(\cdot,x)$. Preserving the sign of the proof of Theorem \ref{geometry}. (page \thePage) define the set \[X=\bigcup_{\n}H_n\left([0,1]\times S_F^p\right)\cup S_F^p.\] Note that sets $S_{\!F}^p$ (Theorem \ref{solset}.) and $H_n\left([0,1]\times S_F^p\right)$ (continuity of homotopy) are compact and that $H_n\left([0,1]\times S_F^p\right)\subset S^p_n$ for every $\n$. By \eqref{beta} it follows that $X$ is compact. \par For every $\eps>0$ we can provide such a function $g_{n_\eps}\colon[0,1]\times S_F^p\times X\to C(I,E)$ that
\begin{itemize}
\item[(i)] $g_{n_\eps}$ is continuous and compact,
\item[(ii)] for every $(s,y)\in[0,1]\times S_F^p$ the equation $x=g_{n_\eps}(s,y,x)$ possesses a nonempty acyclic solution set contained in the $\eps$-neighbourhood $B(S_F^p,\eps)$,
\item[(iii)] $y=g_{n_\eps}(0,y,y)$ for $s=0$ and for every $y\in S_F^p$,
\item[(iv)] for $s=1$ there exists $x_0\in X$ such that $x_0=g_{n_\eps}(1,y,x_0)$ for every $y\in S_F^p$.
\end{itemize}
Indeed, let
\begin{equation}\label{gneps}
g_{n_\eps}(s,y,x)(t)=
\begin{cases}
y(t)&\mbox{for }t\in[0,(1-s)T],\\
h(t)+\int\limits_0^{(1-s)T}k(t,\tau)w_y(\tau)\,d\tau+\int\limits_{(1-s)T}^tk(t,\tau)f_{n_\eps}(\tau,x(\tau))\,d\tau&\mbox{for }t\in[(1-s)T,T],
\end{cases}
\end{equation}
where $y=h+V(w_y)$ for $w_y\in N_F^p(y)$.\par The solution set of the equation $x=g_{n_\eps}(s,y,x)$ is acyclic for any pair $(s,y)$, because it is a singleton in view of Lemma \ref{equation}. Definition \eqref{gneps} guarantees that this solution belongs to the set $S_{F_{n_\eps}}^p$. Thus, if $n_\eps\in\varmathbb{N}$ is sufficiently large then above-mentioned solution belongs to $\eps$-areola of the set $S_{\!F}^p$ - according to \eqref{beta}. Property (iii) follows immediately from definition \eqref{gneps}. If $s=1$, then equation $x=g_{n_\eps}(s,y,x)$ takes the form:
\begin{equation}\label{gwiazdka}
x(t)=h(t)+\int_0^tk(t,\tau)f_{n_\eps}(\tau,x(\tau))\,d\tau,\;t\in I.
\end{equation}
Reapplying of Lemma \ref{equation}. ensures the existence of a unique solution $x_0\in X$ to the equation \eqref{gwiazdka}. It is clear that this solution does not depend on the choice of the points $y\in S_F^p$.\par Let us move on to the continuity of $g_{n_\eps}$. Let $(s_k,y_k,x_k)\to(s,y,x)$ in $[0,1]\times S_F^p\times X$ as $k\to\infty$. Fix our attention on the case $s_k\searrow s$. It is easy to see that the sequence of functions $(w_{y_k})_\K$, corresponding to the choice of points $y_k\in S_F^p$, is relatively weakly compact in the space $L^p(I,E)$. If one takes into account the convergence $V(w_{y_k})\rightharpoonup V(w_y)$ and the injectivity of operator $V$ (Lemma \ref{lem2}.), it becomes clear that $w_{y_k}\rightharpoonup w_y$ (with precision of a subsequence). \par Denote by $V_{(1-s)T}\colon L^p(I,E)\to C(I,E)$ a linear continuous mapping defined by the formula
\[V_{(1-s)T}(w)(t)=\int_0^{(1-s)T}\!\!\!\!\!\!\!\!k(t,\tau)w(\tau)\,d\tau.\] Clearly $V_{(1-s)T}(w_{y_k})\rightharpoonup V_{(1-s)T}(w_y)$. The uniform convergence $V_{(1-s)T}(w_{y_k})\rightrightarrows V_{(1-s)T}(w_y)$ results from the equicontinuity of the family $\{V_{(1-s)T}(w_{y_k})\}_\K$ and from the following estimation
\begin{align*}
\beta(\{V_{(1-s)T}(w_{y_k})(t)\}_\K)&\<2\int_0^{(1-s)T}||k(t,\tau)||_{{\mathcal L}}\beta(\{w_{y_k}(\tau)\}_\K)d\tau\<2\int_0^{(1-s)T}||k(t,\tau)||_{{\mathcal L}}\eta(\tau)\beta(\{y_k(\tau)\}_\K)d\tau.
\end{align*}
\par Recall that we can associate with the compact set $K=\bigcup_\K x_k(I)\cup x(I)$ such a mapping $\mu_K\in L^p(I,\R{})$ that $|f_{n_\eps}(t,x_1)-f_{n_\eps}(t,x_2)|\<\mu_K(t)|x_1-x_2|$ for $t\in I$ and for all $x_1,x_2\in K$. Using the above findings, we can estimate:
{\allowdisplaybreaks
\begin{align*}
\sup_{t\in I}|g_{n_\eps}(s_k,y_k,x_k)(t)&-g_{n_\eps}(s,y,x)(t)|\<\sup_{t\in I}\left|\int_0^{(1-s)T}\!\!\!\!k(t,\tau)w_{y_k}(\tau)d\tau-\int_0^{(1-s)T}\!\!\!\!k(t,\tau)w_y(\tau)d\tau\right|\\&+\sup_{t\in I}\left|\int_{(1-s_k)T}^{(1-s)T}\!\!\!\!k(t,\tau)w_{y_k}(\tau)d\tau\right|+\sup_{t\in I}\left|\int_{(1-s_k)T}^{(1-s)T}\!\!\!\!k(t,\tau)f_{n_\eps}(\tau,x_k(\tau))d\tau\right|\\&+\sup_{t\in I}\left|\int_{(1-s)T}^t\!\!\!\!k(t,\tau)f_{n_\eps}(\tau,x_k(\tau))d\tau-\int_{(1-s)T}^t\!\!\!\!k(t,\tau)f_{n_\eps}(\tau,x(\tau))d\tau\right|\\&\<||V_{(1-s)T}(w_{y_k})-V_{(1-s)T}(w_y)||+\sup_{t\in I}||k(t,\cdot)||_q\left(\int_{(1-s_k)T}^{(1-s)T}|w_{y_k}(\tau)|^pd\tau\right)^\frac{1}{p}\\&+\sup_{t\in I}||k(t,\cdot)||_q\left(\int_{(1-s_k)T}^{(1-s)T}\!\!\!\!\!\!\!\!\!\!\!|f_{n_\eps}(\tau,x_k(\tau))|^pd\tau\right)^\frac{1}{p}\!\!\!+\sup_{t\in I}\!\int_{(1-s)T}^t\!\!\!\!\!\!\!\!\!||k(t,\tau)||_{{\mathcal L}}\mu_K(\tau)|x_k(\tau)\!-\!x(\tau)|d\tau\\&\<||V_{(1-s)T}(w_{y_k})-V_{(1-s)T}(w_y)||+\sup_{t\in I}||k(t,\cdot)||_q\left(2\left(\int_{(1-s_k)T}^{(1-s)T}\!\!\!\!\!\mu(\tau)^pd\tau\right)^\frac{1}{p}+||\mu_K||_p||x_k-x||\right),
\end{align*}}
i.e. $||g_{n_\eps}(s_k,y_k,x_k)-g_{n_\eps}(s,y,x)||\to 0$ as $k\to\infty$. Therefore $g_{n_\eps}$ is continuous and consequently compact.\par Observe that the set-valued operator ${\mathcal F}\colon X\map C(I,E)$ given by the formula ${\mathcal F}=h+V\circ N_F^p$ has a compact range ${\mathcal F}(X)$, because it is an upper semicontinuous map with compact values. The singlevalued map $i\colon X\to C(I,E)$ such that $i(x)=x$ is on the other hand proper continuous. In view of Theorem 4.3. in \cite{gelman} we infer that the set $\{x\in X\colon i(x)\in {\mathcal F}(x)\}=S_{\!F}^p$ is acyclic.  
\end{proof}
To be able to use the argument contained in the proof of Theorem \ref{acyclic}. in the context of integral equations it is sufficient to know that the Niemytskij operator $N_{\!f}^p$, given by $N_{\!f}^p(x)=f(\cdot,x(\cdot))$, is at least weakly continuous. Therefore, we can formulate the following proposal.
\begin{corollary}
Suppose that kernel $k$ fulfills conditions $(k_5)$-$(k_6)$. Assume that the mapping $f$ satisfies either $(f_1)$-$(f_3)$ and $(f_4)$ or $(f_1)$-$(f_3)$ and $(f_5)$, provided the dual space $E^*$ has the Radon-Nikod\'ym property. Then the set $S_{\!f}^p$ of continuous solutions to integral equation \eqref{rownanie} is acyclic.
\end{corollary}
\par We conclude the current section with the following formulation of continuous dependence of solutions to integral inclusion on nonlinear perturbations.
\begin{proposition}\label{solutionsetmap}
Assume that conditions $(F_1)$-$(F_5)$ and $(k_5)$-$(k_6)$ are satisfied. Then the solution set map $S_F^p\colon C(I,E)\map C(I,E)$ given by the formula 
\begin{equation}\label{solsetmap}
S_F^p(h)=\left\{x\in C(I,E)\colon x\in h+V\circ N_F^p(x)\right\}
\end{equation}
is upper semicontinuous.
\end{proposition}
\begin{proof}
Suppose the sequence $(h_n)_\n$ converges uniformly to $h$ and $x_n\in S_F^p(h_n)$ for $\n$. We claim that sequence $(x_n)_\n$ is relatively compact. Indeed, for every $t,\tau\in I$ and $\n$ the following inequalities hold
\[|x_n(t)-x_n(\tau)|\<\sup_\n|h_n(t)-h_n(\tau)|+\sup_{t\in I}||k(t,\cdot)||_q\left(\int_t^\tau\mu(s)^p\,ds\right)^\frac{1}{p}+||k(\tau,\cdot)-k(t,\cdot)||_q||\mu||_p.\]
From these results the equicontinuity of the family $\{x_n\}_\n$. Assume that $x_n=h_n+V(w_n)$ and $w_n\in N_F^p(x_n)$. Using equicontinuity of $\{h_n\}_\n$, condition $(F_5)$ and Heinz inequality (Theorem \ref{3}.) we arrive at
\begin{align*}
\beta(\{x_n(t)\}_\n)&=\beta\left(\left\{h_n(t)+\int_0^t\!\!k(t,s)w_n(s)\,ds\right\}_\n\right)\<\beta(\{h_n(t)\}_\n)+2\int_0^t||k(t,s)||_{{\mathcal L}}\beta(\{w_n(s)\}_\n)\,ds\\&\<2\int_0^t||k(t,s)||_{{\mathcal L}}\beta(F(s,\{x_n(s)\}_\n))\,ds\<2\int_0^t||k(t,s)||_{{\mathcal L}}\eta(s)\beta(\{x_n(s)\}_\n)\,ds
\end{align*}
for $t\in I$. Thus, for any $t\in I$ 
\[\beta(\{x_n(t)\}_\n)^p\<2^p\left(\sup_{t\in I}||k(t,\cdot)||_q\right)^p\int_0^t\eta(s)^p\beta(\{x_n(s)\}_\n)^p\,ds\] and consequently $\beta(\{x_n(t)\}_\n)=0$ for every $t\in I$, by Gronwall inequality. Ipso facto, Arzel\`a theorem implies existence of a subsequence $(x_{k_n})_\n$, which converges to some $x$ in $C(I,E)$. We are able to distinguish a weakly convergent in the space $L^p(I,E)$ subsequence $(w_{k_n})_\n$ such that $w_{k_n}\rightharpoonup w\in N_F^p(x)$. Based on the weak continuity of Volterra oparator $V$ we infer that $x_{k_n}=h_{k_n}+V(w_{k_n})\rightharpoonup h+V(w)$ as $n\to\infty$. The uniqueness of a weak limit means that $x=h+V(w)$. So finally $x\in S_F^p(h)$. 
\end{proof}
\begin{corollary}\label{wn3}
The set-valued map $S_F^p\colon C(I,E)\map C(I,E)$ remains upper semicontinuous, if we replace assumption $(F_5)$ by the assumption $(k_7)$.
\end{corollary}
\section{Applications}
\par Subsequent statements address the question of the existence of periodic solutions to Volterra integral inclusions. In their proofs are applied earlier results on the topological and geometric structure of the solution set. 
\begin{theorem}\label{pervol1}
Let $p\in[1,\infty)$, while the space $E$ is reflexive for $p\in(1,\infty)$. Assume that conditions $(k_1)$-$(k_4)$ and $(F_1)$-$(F_5)$ hold with
\begin{equation}\label{condwithp}
\sup\limits_{t\in I}||k(t,\cdot)||_q||\eta||_p<2^{\frac{2}{p}-3}e^{-\frac{1}{p}}.
\end{equation}
Then there exists a continuous function $h\colon I\to E$, for which the integral inclusion \eqref{inclusion} possesses a $T$-periodic solution.
\end{theorem}
\begin{proof}
Suppose there is a function $\varrho\in L^p(I,\R{})$ and a number $\omega>0$ such that for all bounded $\Omega\subset E$ and for almost all $t\in I$ the following inequality holds
\begin{equation}\label{varrho}
\beta(F(t,\Omega))\<e^{\omega(t-T)}\varrho(t)\beta(\Omega).
\end{equation}
We claim that, if $\omega T>(1-p^{-1})\ln2$ and
\begin{equation}\label{<1}
\sup\limits_{t\in I}||k(t,\cdot)||_q||\varrho||_p<2^{\frac{1}{p}-2}\left((1-p)\ln(2)+p\omega T\right)^{\frac{1}{p}},
\end{equation}
then the thesis of the current theorem is true. \par First observe that condition \eqref{<1} is reasonable, because $2^{\frac{1}{p}-2}\left((1-p)\ln(2)+p\omega T\right)^{\frac{1}{p}}>0$ provided $\omega T>(1-p^{-1})\ln2$. Let $\{U(t)\}_{t\in I}\subset{\mathcal L}(E)$ be a family of operators possessing the following properties: 
\begin{itemize}
\item[$(U_1)$] $U(0)=id_E$,
\item[$(U_2)$] $\{U(t)\}_{t\in I}$ is uniformly continuous, i.e. $||U(t)-U(\tau)||_{{\mathcal L}}\to 0$ as $t\to\tau$,
\item[$(U_3)$] $\{U(t)\}_{t\in I}$ is uniformly exponentially stable in the sense that $||U(t)||_{{\mathcal L}}\<e^{-\omega t}$ for every $t\in I$.
\end{itemize}
Next define a Poncar\'e-type operator $P_t\colon E\map E$ by the formula
\[P_T(x)=ev_T(S_F^p(U(\cdot)x)),\]
where $ev_T\colon C(I,E)\to E$ denotes the evaluation map at the point $T$, whereas $S_F^p\colon C(I,E)\map C(I,E)$ is the solution set map given by \eqref{solsetmap}. Note that if $k(T,\cdot)\equiv 0$, then every solution of \eqref{inclusion} is $T$-periodic provided the inhomogeneity $h$ is also $T$-periodic. So let us assume that $M=||k(T,\cdot)||_q||\mu||_p>0$ and take the radius $R=\left(1-e^{-\omega T}\right)^{-1}\cdot M$. Now, it is clear that $P_T(D(0,R))\subset D(0,R)$. In view of Theorem \ref{solset}., Theorem \ref{acyclic}. and Proposition \ref{solutionsetmap}. the set-valued map $S_{\!F}^p$ is acyclic. Therefore operator $P_T\colon D(0,R)\map D(0,R)$ is admissible (\cite[Th. 40.6]{gorn}).\par Let $\w{\beta}\colon{\mathcal B}\to\R{}$ be the MNC defined by \eqref{sequentialbeta}. Assume that 
\begin{equation}\label{kond}
\w{\beta}(\Omega)\<\w{\beta}(P_T(\Omega))
\end{equation}
for some $\Omega\in{\mathcal B}$. Then there are countable subsets $\{y_n\}_\n\subset C(I,E)$ and $\{x_n\}_\n\subset\Omega$ such that $y_n\in S_F^p(U(\cdot)x_n)$ and $\w{\beta}(P_T(\Omega))=\beta(\{y_n(T)\}_\n)$. We have the following estimation
\begin{align*}
|y_n(t)-y_n(\tau)|\<||U(t)-U(\tau)||_{{\mathcal L}}\sup_\n|x_n|+\sup_{t\in I}||k(t,\cdot)||_q\left(\int_t^\tau\mu(s)^p\,ds\right)^\frac{1}{p}\!+\!||k(\tau,\cdot)-k(t,\cdot)||_q||\mu||_p
\end{align*} 
It is obvious that the family $\{y_n\}_\n$ is equicontinuous, inter alia, due to the property $(U_2)$. Let us introduce an auxiliary function $f\colon I\to\R{}$ such that $f(t)=\beta(\{y_n(t)\}_\n)$. Reffering to the inequality \eqref{varrho} and properites $(U_1)$, $(U_3)$ we are able to estimate:
\begin{align*}
f(t)&\<\beta(U(t)\{x_n\}_\n)+\beta\left(\left\{\int_0^tk(t,s)w_n(s)ds\right\}_\n\right)\<||U(t)||_{{\mathcal L}}\beta(\{x_n\}_\n)+2\int_0^t\beta(k(t,s)\{w_n(s)\}_\n)ds\\&\<e^{-\omega t}\beta(\{y_n(0)\}_\n)+2\int_0^t||k(t,s)||_{{\mathcal L}}e^{\omega(s-t)}\varrho(s)\beta(\{y_n(s)\}_\n)ds\\&\<e^{-\omega t}f(0)+2\sup\limits_{\tau\in I}||k(\tau,\cdot)||_q\left(\int_0^te^{p\omega(s-t)}\varrho(s)^pf(s)^pds\right)^{\frac{1}{p}},
\end{align*}
where $w_n\in N_F^p(y_n)$. Seeing that the mapping $\R{}_+\ni x\mapsto x^p\in\R{}_+$ is monotone and convex, we conclude that
\begin{align*}
f(t)^p\<\left(e^{-\omega t}f(0)+2B\left(\int_0^t\!\!e^{p\omega(s-t)}\varrho(s)^pf(s)^pds\right)^{\frac{1}{p}}\right)^p\!\<\frac{1}{2}2^pf(0)^pe^{-p\omega t}+\frac{1}{2}2^{2p}B^pe^{-p\omega t}\!\!\int_0^t\!\!e^{p\omega s}\varrho(s)^pf(s)^pds
\end{align*}
for every $t\in I$, where $B=\sup_{\tau\in I}||k(\tau,\cdot)||_q$. Denote the right-hand side of the above inequality as the function $g$ of the variable $t$. Then
\[g'(t)=-p\omega g(t)+2^{2p-1}B^p\varrho(t)^pf(t)^p\<-p\omega g(t)+2^{2p-1}B^p\varrho(t)^pg(t)\]
for almost all $t\in I$. The solution of the following initial problem
\[\begin{cases}
\dot{x}(t)=-p\omega x(t)+2^{2p-1}B^p\varrho(t)^px(t)&\mbox{for almost all }t\in I\\
x(0)=g(0)
\end{cases}\]
majorizes function $g$. Therefore we have this estimation:
\[f(T)^p\<g(T)\<g(0)e^{-p\omega T+2^{2p-1}B^p||\varrho||_p^p}\]
and as a result
\[f(T)\<2^{\frac{p-1}{p}}e^{-\omega T+p^{-1}2^{2p-1}B^p||\varrho||_p^p}f(0).\]
There are two possibilities: either $f(0)=0$ or $f(0)>0$. If $f(0)>0$, then
\[\w{\beta}(P_T(\Omega))=f(T)<f(0)=\beta(\{x_n\}_\n)\<\w{\beta}(\Omega),\]
because $2^{\frac{p-1}{p}}e^{-\omega T+p^{-1}2^{2p-1}B^p||\varrho||_p^p}<1$ in view of the assumption \eqref{<1}. We get a contradiction with the condition \eqref{kond}. Therefore the equality $f(0)=0$ must hold. In this case we have $f(T)=\w{\beta}(P_T(\Omega))=0$, which means that $\w{\beta}(\Omega)=0$. Since the MNC $\w{\beta}$ is regular, the set $\Omega$ must be relatively compact in $E$. This proves that the operator $P_T\colon D(0,R)\map D(0,R)$ is $\w{\beta}$-condensing. Hence, in view of Theorem \ref{4}., there is a point $x_0\in D(0,R)$ such that $x_0\in P_T(x_0)$. This point generates a $T$-periodic solution to inclusion \eqref{inclusion} with the perturbation component $h\colon I\to E$ given by $h(t)=U(t)x_0$.\par It now remains to prove the thesis in the context of the presupposed inequality \eqref{condwithp}. To this end define an auxiliary map $\xi\colon\R{}\to\R{}$ such that $\xi(\omega)=2^{\frac{1}{p}-2}\left((1-p)\ln(2)+p\omega T\right)^{\frac{1}{p}}e^{-\omega T}$ and the function $\varrho\in L^p(I,\R{})$ by the formula $\varrho(t)=e^{\omega_0(T-t)}\eta(t)$, where $\omega_0$ is a fixed number given by $\omega_0=\frac{1+(p-1)\ln 2}{pT}$. Thanks to the assumption $(F_5)$ inequality \eqref{varrho} is satisfied. At the same time
\[\sup_{t\in I}||k(t,\cdot)||_q||\varrho||_p\<\sup_{t\in I}||k(t,\cdot)||_q||\eta||_pe^{\omega_0T}<\xi(\omega_0)e^{\omega_0T}=2^{\frac{1}{p}-2}((1-p)\ln(2)+p\omega_0T)^{\frac{1}{p}},\]
which means that inequality \eqref{<1} is also satisfied. By virtue of the foregoing justified claim the existence of periodic solutions is established. Observe that estimation \eqref{condwithp} is the best possible in the sense that $2^{\frac{2}{p}-3}e^{-\frac{1}{p}}=\xi(\omega_0)=\max\left\{\xi(\omega)\colon\omega\in\left(\frac{(p-1)\ln 2}{pT},\infty\right)\right\}$.
\end{proof}
\begin{remark}
Observe that the assumption \eqref{condwithp} is stronger than assumption $(E_1)$, which supports Theorem \ref{geometry}. Hence the Poincar\'e operator used in the above proof belongs in fact to the class of so-called decomposable mappings (see \cite{gorn}).
\end{remark}
\begin{theorem}\label{pervol2}
Let $p\in[1,\infty)$, while the space $E$ is reflexive for $p\in(1,\infty)$. Assume that conditions $(F_1)$-$(F_4')$ and $(k_1)$-$(k_4)$ together with $(k_7)$ are satisfied. Then there exists a continuous function $h\colon I\to E$, for which the integral inclusion \eqref{inclusion} possesses a $T$-periodic solution.
\end{theorem}
\begin{proof}
Assume that a family of operators $\{U(t)\}_{t\in I}\subset{\mathcal L}(E)$ satisfies conditions of the form:
\begin{itemize}
\item[$(U_1)$] $U(0)=id_E$,
\item[$(U_2')$] $\{U(t)\}_{t\in I}$ is strongly continuous, i.e. the map $I\ni t\mapsto U(t)x\in E$ is continuous for every $x\in E$,
\item[$(U_3')$] $||U(T)||_{{\mathcal L}}<1$.
\end{itemize}
Conditions $(U_2')$-$(U_3')$ are obviously less restrictive than previously formulated assumptions $(U_2)$-$(U_3)$. The uniform boundedness principle together with condition $(U_2')$ imply uniform boundedness of the family $\{U(t)\}_{t\in I}$. In view of Corollary \ref{wn3}. the multivalued map $E\ni x\mapsto S_F^p(U(\cdot)x)\subset C(I,E)$ must be upper semicontinuous. Moreover, this map possesses nonempty compact and acyclic values (as it was proven in Corollary \ref{solset2}. and Corollary \ref{geometry2}.). \par Let $r=||U(T)||_{{\mathcal L}}$ and $M=||k(T,\cdot)||_q||\mu||_p$. Then $R=(1-r)^{-1}\cdot M>0$. We claim that the Poincar\'e-type operator $P_T\colon D(0,R)\map D(0,R)$, such that $P_T(x)=ev_T(S_F^p(U(\cdot)x))$, is condensing relative to MNC $\w{\beta}$. Let $\Omega$ be a bounded subset of $E$ and $\w{\beta}(P_T(\Omega))=\beta(\{U(T)x_n+V(w_n)(T)\}_\n)$. Then, applying Theorem \ref{3}. and condition $(k_7)$, we obtain a sequence of inequalities:
\begin{align*}
\w{\beta}(P_T(\Omega))\<||U(T)||_{{\mathcal L}}\beta(\{x_n\}_\n)+2\int_0^T\beta(k(T,s)\{w_n(s)\}_\n)\,ds\<||U(T)||_{{\mathcal L}}\w{\beta}(\Omega).
\end{align*}
Therefore, from $(U_3')$ it follows that $\w{\beta}(P_T(\Omega))<\w{\beta}(\Omega)$, provided $\w{\beta}(\Omega)>0$.
\par Summarizing: $\w{\beta}$-condensing and strongly admissible operator $P_T$ possesses a fixed point $x_0\in P_T(x_0)$ (following Theorem \ref{4}.), which guarantees the existence of a $T$-periodic solution to the problem \eqref{inclusion}, with the inhomogeneity $h\in C(I,E)$ such that $h(t)=U(t)x_0$.   
\end{proof}
\par The article concludes by formulating a rather simple observation regarding the existence of periodic solutions to the following, so called Hammerstein integral inclusion
\begin{equation}\label{math_H}
x(t)\in h(t)+\int_0^Tk(t,s)F(s,x(s))\,ds,\;t\in I.
\end{equation}
This time, kernel $k$ is a mapping defined on the whole product $I\times I$. Observe that the Volterra inclusion \eqref{inclusion} is nothing more than a special case of \eqref{math_H}. The next theorem significantly generalizes the result from \cite{piet}.
\begin{theorem}\label{perham1}
Let $p\in[1,\infty)$ and $E$ be a reflexive space for $p\in(1,\infty)$. Suppose that the multivalued map $F$ satisfies assumptions $(F_1)$-$(F_5)$, while the function $k\colon I\times I\to{\mathcal L}(E)$ conditions
\begin{itemize}
\item[{\em $(k_5')$}] for every $t\in I$, $k(t,\cdot)\in L^q(I,{\mathcal L}(E))$, where $p^{-1}+q^{-1}=1$,
\item[{\em $(k_6')$}] the mapping $I\ni t\mapsto k(t,\cdot)\in L^q(I,{\mathcal L}(E))$ is continuous and $T$-periodic.
\end{itemize}
Suppose futher that the following inequality holds
\begin{equation}\label{nierzsup}
2\sup\limits_{t\in I}||k(t,\cdot)||_q||\eta||_p<1.
\end{equation}
Then the integral inclusion \eqref{math_H} possesses a continuous $T$-periodic solution for any $T$-periodic function $h\in C(I,E)$.
\end{theorem}
\begin{proof}
Let $h\in C(I,E)$ be any function with period $T$. Denote by $V_T\colon L^p(I,E)\to C(I,E)$ a mapping of the following form:
\[V_T(w)(t)=\int_0^Tk(t,s)w(s)\,ds,\;t\in I.\]
It is clear that the fixed points of the operator ${\mathcal F}_T=h+V_T\circ N_F^p$ determine the solutions of inclusion \eqref{math_H}. These points are located in the ball of radius $R=||h||+\sup_{t\in I}||k(t,\cdot)||_q||\mu||_p$. Substitution of the integral operator $V_T$ in place of the Volterra operator $V$ in the proof of Lemma \ref{calFlem}. leads to the observation that the set-valued map ${\mathcal F}_T\colon D(0,R)\map D(0,R)$ is upper semicontinuous and possesses nonempty convex compact values. The main inconvenience that prevents adopting all the arguments contained in the proof of Lemma \ref{calFlem}. is that 
\[\inf_{L\geqslant 0}\sup_{t\in I}e^{-Lt}\left(\int_0^T(\eta(s)e^{Ls})^p\,ds\right)^\frac{1}{p}\geqslant\inf_{L\geqslant 0}e^{-L0}\left(\int_0^T(\eta(s)e^{Ls})^p\,ds\right)^\frac{1}{p}\geqslant\left(\int_0^T\eta(s)^p\,ds\right)^\frac{1}{p}=||\eta||_p,\] while the term $||\eta||_p$ is significantly greater than zero, provided assumption $(F_5)$ is nontrivial. The analysis of this proof, however, indicates that the adoption of the hypothesis \eqref{nierzsup} authorizes the condensing of the operator ${\mathcal F}_T$ relative to MNC $\nu_0$ described as follows: \[\nu_0(\Omega)=\max\limits_{D\in\Delta(\Omega)}\left(\sup\limits_{t\in I}\beta(D(t)),\modulus_C(D)\right).\] \par To confirm this observation, suppose that measures of noncompactness $\nu_0(\Omega)$ and $\nu_0({\mathcal F}_T(\Omega))$ are attained respectively on countable sets $\{y_n\}_\n$ and $\{v_n\}_\n$ and make a small adjustment in the sequence of inequalities bearing the number \eqref{20}, namely
\begin{align*}
\beta(\{v_n(t)\}_\n)&=\beta(\{h(t)+V_T(w_n)(t)\}_\n)\<2\!\!\int_0^T\!\!||k(t,s)||_{{\mathcal L}}\beta(F(s,\{u_n(s)\}_\n))ds\\&\<2\!\!\int_0^T\!\!||k(t,s)||_{{\mathcal L}}\eta(s)\beta(\{u_n(s)\}_\n)ds\<\sup\limits_{t\in I}\beta(\{u_n(t)\}_\n)\,2\int_0^T||k(t,s)||_{{\mathcal L}}\eta(s)ds.
\end{align*}
Hence
\begin{align*}
\sup\limits_{t\in I}\beta(\{v_n(t)\}_\n)&\<\sup\limits_{t\in I}\beta(\{u_n(t)\}_\n)2\sup_{t\in I}\int_0^T||k(t,s)||_{{\mathcal L}}\eta(s)ds\\&\<\sup_{t\in I}\beta(\{u_n(t)\}_\n)2\sup_{t\in I}||k(t,\cdot)||_q||\eta||_p\<2\sup\limits_{t\in I}||k(t,\cdot)||_q||\eta||_p\sup_{t\in I}\beta(\{y_n(t)\}_\n).
\end{align*}
If we assume that $\nu_0(\Omega)\<\nu_0({\mathcal F}_T(\Omega))$ then $\sup_{t\in I}\beta(\{y_n(t)\}_\n)=0$, otherwise hypothesis \eqref{nierzsup} leads to the contradiction. On the other hand $\modulus_C(\{v_n\}_\n)=0$, which means that $\modulus_C(\{y_n\}_\n)=0$. In conclusion $\Omega$ must be relatively compact and ${\mathcal F}_T$ is $\nu_0$-condensing operator.
\par Suppose that $x\in D(0,R)$ is a fixed point of the map ${\mathcal F}_T$. Then there exists $w\in N_F^p(x)$ satisfying equation $x=h+V_T(w)$. The assumption $(k_6')$ implies that $k(0,s)w(s)=k(T,s)w(s)$ for almost all $s\in I$, and therefore
\[x(0)=h(0)+\int_0^Tk(0,s)w(s)\,ds=h(T)+\int_0^Tk(T,s)w(s)\,ds=x(T).\] Thus, it is clear that in the case of $T$-periodicity of the mappings $h$ and $k(t,\cdot)$ every solution to the inclusion \eqref{math_H} is $T$-periodic.
\end{proof}
The following thesis is fully legible in the context of Theorem \ref{perham1}. and Corollary \ref{solset2}.
\begin{theorem}\label{perham2}
Let $p\in[1,\infty)$, while the space $E$ is reflexive for $p\in(1,\infty)$. Assume that conditions $(F_1)$-$(F_4')$ and $(k_5')$-$(k_6')$ are satisfied, together with
\begin{itemize}
\item[$(k_7')$] the operator $k(t,s)$ is completely continuous for all $(t,s)\in I\times I$.
\end{itemize}
Then problem \eqref{math_H} possesses a continuous $T$-periodic solution for any $T$-periodic function $h\in C(I,E)$.
\end{theorem}
\begin{example}
Let $f\colon I\to{\mathcal L}(E)$ be a continuous and $T$-periodic function, while $g\in L^q(I,{\mathcal L}(E))$. Then kernel $k\colon I\times I\to{\mathcal L}(E)$ defined by
\[k(t,s)=f(t)\circ g(s),\;(t,s)\in I\times I\]
satisfies conditions $(k_5')$-$(k_6')$.
\end{example}


\begin{thebibliography}{12}
\bibitem[1] {agarwal} R. P. Agarwal, D. O'Regan, {\it The solution set of integral inclusions on the half line}, Integral and Integrodifferential Equations: Theory, Methods and Applications, Gordon and Breach Publishers, 2000, 1-7. 
\bibitem[2] {sadovski} R. Akhmerov, M. Kamenskii, A. Potapov, A. Rodkina, B. Sadovskii, {\it Measures of noncompactness and condensing operators}, Birkh\"{a}user Verlag, 1992. 
\bibitem[3] {aubin} J. P. Aubin, A. Cellina, {\it Differential Inclusions}, Springer, Berlin, 1984.
\bibitem[4] {bothe} D. Bothe, {\it Multivalued perturbations of m-accretive differential inclusions}, Isr. J. Math. 108 (1998), 109-138.
\bibitem[5] {bugaj} D. Bugajewski, {\it On the existence of weak solutions of integral equations in Banach spaces}, Comment. Math. Univ. Carolin. 35,1 (1994), 35-41.
\bibitem[6] {bulg1} A. I. Bulgakov, L. N. Ljapin, {\it Some properties of the set of solutions of a Volterra-Hammerstein integral inclusions} (in Russian), Differentsialnye Uravneniya 14 (1978), no. 8, 1465-1472.
\bibitem[7] {bulg2} A. I. Bulgakov, L. N. Ljapin, {\it Integral inclusions with a functional operator} (in Russian), Differentsialnye Uravneniya 15 (1979), no. 5, 878-884.
\bibitem[8] {deimling} K. Deimling, {\it Multivalued differential equations}, Walter de Gruyter, Berlin-New York, 1992.
\bibitem[9] {uhl} J. Diestel, J. J. Uhl, {\it Vector measures}, Math. Surveys Monographs, vol.15, Amer. Math. Soc., Providence, RI, 1977.
\bibitem[10] {gelman} B. D. Gel'man, {\it Topological properties of the set of fixed points of a multivalued map}, Mat. Sb. 188:12 (1997), 1761-1782. 
\bibitem[11] {gorn} L. G\'orniewicz, {\it Topological fixed point theory of multivalued mappings}, Second ed., Springer, Dordrecht, 2006.
\bibitem[12] {heinz} H. P. Heinz, {\it On the behaviour of measures of noncompactness with respect to differentiation and integration of vector-valued functions}, Nonlinear Anal. 7 (1983), 1351-1371.
\bibitem[13] {hyman} D. M. Hyman, {\it On decreasing sequences of compact absolute retracts}, Fund. Math. 64 (1969), 91-97.
\bibitem[14] {zecca} M. Kamenskii, V. Obukhovskii, P. Zecca, {\it Condensing multivalued maps and semilinear differential inclusions in Banach spaces}, Nonlinear Anal. Appl., vol.7, Walter de Gruyter, Berlin-New York, 2001.
\bibitem[15] {lesniak} K. Le\'sniak, {\it Infinite iterated function systems: A multivalued approach}, Bull. Pol. Ac. Sci.: Math 52 no.1 (2004), 1-8.
\bibitem[16] {reg} D. O'Regan, {\it Volterra and Urysohn integral equations in Banach spaces}, J. Appl. Math. Stoch. Anal. 11:4 (1998), 449-464.
\bibitem[17] {pre} D. O'Regan, R. Precup, {\it Existence criteria for integral equations in Banach spaces}, J. of Inequal. \& Appl. 6 (2001), 77-97.
\bibitem[18] {piet} R. Pietkun, {\it Periodic and almost periodic solutions of integral inclusions.}, Bull. Pol. Ac. Sci.: Math 58 no.3 (2010), 239-245.
\bibitem[19] {szufla} S. Szufla, {\it On Volterra integral equations in Banach spaces}, Funkcial. Ekvac. 20 (1977), 247-258.
\bibitem[20] {talagrand} M. Talagrand, {\it Pettis integral and measure theory}, Memoirs of the AMS, vol.15, no.307, Providence, RI, 1984.
\bibitem[21] {umanski} Ya. I. Umanski\v\i, {\it On a property of the solution set of differential inclusions in a Banach space} (in Russian), Differentsialnye Uravneniya 28 (1992), no. 8, 1346-1351.
\bibitem[22] {ulger} A. \"{U}lger, {\it Weak compactness in $L^1(\mu,X)$}, Proc. Amer. Math. Soc. 113 (1991), 143-149.
\bibitem[23] {vath} M. V\"ath, {\it Volterra and integral equations of vector functions}, M. Dekker, Inc., New York-Basel, 2000.
\end{thebibliography}
\end{document}